\documentclass{article}

\usepackage[noend]{algpseudocode}
\usepackage{algorithm}

\newcommand*\Let[2]{\State #1 $\gets$ #2}
\newcommand{\Hess}{\operatorname{Hess}}

\usepackage{constants}

\usepackage[margin=3.2cm]{geometry}

\usepackage{setspace}

\usepackage{graphics}
\usepackage{graphicx}

\usepackage{docmute}

\usepackage[utf8]{inputenc}

\usepackage{mathtools}

\usepackage{amsmath}

\usepackage{amsfonts}

\usepackage{amssymb}

\usepackage{amsthm}
\usepackage{aliascnt}

\usepackage{makeidx}
\makeindex

\usepackage{textcomp}
\usepackage[sb]{libertine}
\usepackage[varqu,varl]{zi4}%
\usepackage[libertine,bigdelims,vvarbb]{newtxmath}
\usepackage[supstfm=libertinesups,supscaled=1.2,raised=-.13em]{superiors}
\usepackage{bm}
\useosf

\usepackage[breaklinks,linktocpage]{hyperref}
\usepackage[hyperpageref]{backref}

\usepackage{url}
\usepackage{breakurl}

\usepackage{tikz-cd}

\usepackage{enumitem}
\usepackage[UKenglish]{babel}
\usepackage[makeroom]{cancel}
\usepackage{pstricks-add}

\definecolor{shadecolor}{rgb}{0.88,0.91,0.95}       
\usepackage{tcolorbox}

\usepackage{dynkin-diagrams}
\usetikzlibrary{backgrounds}

\usepackage{booktabs}

\usepackage{fancyhdr}
\newcommand{\R}{\mathbb{R}}

\newcommand{\sign}{\operatorname{sign}}

\newcommand{\spann}{\operatorname{span}}

\renewcommand{\Im}{\operatorname{Im}}

\newcommand{\Ric}{\operatorname{Ric}}

\newcommand{\CP}{\mathbb{CP}}

\renewcommand{\|}[1]{\left| \left| #1 \right| \right|}

\newcommand{\<}{\langle}
\renewcommand{\>}{\rangle}

\newcommand{\grad}{\operatorname{grad}}

\newcommand{\diam}{\operatorname{diam}}

\renewcommand{\d}{\operatorname{d} \!}

\renewcommand{\tilde}[1]{\widetilde{#1}}

\newcommand\Item[1][]{%
  \ifx\relax#1\relax  \item \else \item[#1] \fi
  \abovedisplayskip=0pt\abovedisplayshortskip=0pt~\vspace*{-\baselineskip}}

\usepackage[capitalise]{cleveref}

\numberwithin{equation}{section}

\newtheorem{proposition}[equation]{Proposition}
\crefname{proposition}{Proposition}{Propositions}
\newtheorem*{proposition*}{Proposition}
\crefname{proposition*}{Proposition}{Propositions}

\newtheorem{lemma}[equation]{Lemma}
\crefname{lemma}{Lemma}{Lemmas}

\newtheorem{corollary}[equation]{Corollary}
\crefname{corollary}{Corollary}{Corollaries}

\newtheorem*{corollary*}{Corollary}
\crefname{corollary*}{Corollary}{Corollaries}

\newtheorem{theorem}[equation]{Theorem}
\crefname{theorem}{Theorem}{Theorems}

\crefname{task}{Task}{Task}

\crefname{conjecture}{Conjecture}{Conjectures}

\crefname{algorithm}{Algorithm}{algorithms}

\newtheorem*{theorem*}{Theorem}
\crefname{theorem*}{Theorem}{Theorems}

\crefname{claim}{Claim}{Claims}

\theoremstyle{remark}

\crefname{question}{Question}{Questions}

\newtheorem{definition}[equation]{Definition}
\crefname{definition}{Definition}{Definitions}

\newtheorem{example}[equation]{Example}
\crefname{example}{Example}{Examples}

\newtheorem{remark}[equation]{Remark}
\crefname{remark}{Remark}{Remarks}

\crefname{assumption}{Assumption}{Assumptions}

\crefname{problem}{Problem}{Problems}
\usepackage{pstricks-add}

\DeclareMathOperator*{\argmin}{arg\,min}
\DeclareMathOperator*{\argmax}{arg\,max}

\title{Lower bounds for the reach and applications}
\author{Daniel Platt and Raúl Sánchez Galán}
\date{\today}

\usepackage{todonotes}

\usepackage{pgfplots}
\pgfplotsset{compat=1.15}
\usepackage{mathrsfs}
\usetikzlibrary{arrows}

\begin{document}

\maketitle

\vspace{-0.7cm}

\begin{abstract}
    The \emph{reach} of a submanifold of $\R^N$ is defined as the largest radius of a tubular neighbourhood around the submanifold that avoids self-intersections. 
    While essential in geometric and topological applications, computing the reach explicitly is notoriously difficult. 
    In this paper, we introduce a rigorous and practical method to compute a guaranteed lower bound for the reach of a submanifold described as the common zero-set of finitely many smooth functions, not necessarily polynomials. 
    Our algorithm uses techniques from numerically verified proofs and is particularly suitable for high-performance parallel implementations.

    We illustrate the utility of this method through several applications. 
    Of special note is a novel algorithm for computing the homology groups of planar curves, achieved by constructing a cubical complex that deformation retracts onto the curve--an approach potentially extendable to higher-dimensional manifolds. 
    Additional applications include an improved comparison inequality between intrinsic and extrinsic distances for submanifolds of $\R^N$, lower bounds for the first eigenvalue of the Laplacian on algebraic varieties and explicit bounds on how much smooth varieties can be deformed without changing their diffeomorphism type.
\end{abstract}

\vspace{-0.5cm}

\tableofcontents

\section{Introduction}

Consider a closed smooth submanifold $M \subset \R^N$, defined as the vanishing set $Z(f)$ of a smooth function $f:\mathbb{R}^N \rightarrow \mathbb{R}$. An important geometric quantity associated with $M$ is its \emph{reach}, originally introduced by Federer in \cite[p.419]{Federer1959}. The reach is defined as the largest number $\tau$ such that any point in the ambient space $\R^N$ within distance $\tau$ of $M$ has a unique nearest point on $M$. Intuitively, the reach characterizes how finely one must sample points from a manifold to accurately reconstruct its geometric and topological structure. Computing the reach exactly, however, is notoriously difficult, and thus, reliable lower bounds are of significant practical and theoretical interest.

Previous approaches to computing such lower bounds typically fall into two categories:
probabilistic methods, including estimation techniques involving the medial axis (e.g., \cite{Aamari2019, Attali2009, Cuevas2014}), and
the guaranteed deterministic method from \cite{DiRocco2022,2023DiRocco}.

An important application of an estimation of the reach is the computation of homological and geometric properties from discrete samples.
Sampling strategies from the literature are \cite[Section 1]{DiRocco2022}, the approach used in physics from \cite[Section 3.2]{Douglas2008}, and the semi-algebraic set sampling in \cite{Burgisser2019}, which also includes a review of previous work on this problem.
    
In this article, we present a new algorithm to obtain a provable lower bound for the $L^1$-norm of the gradient $|\nabla f|_1$ on $M=Z(f)$, see \cref{section:lower-bound-for-df}. 
To our knowledge, it is the first method that works for zero sets of functions that are not polynomials.
Also, it can compute a lower bound for the reach in some examples in which the existing algorithms fail due to complexity. 
There are other examples where existing algorithms work, but our algorithm fails. 
These cases are discussed in \cref{remark:algo-comparison}.

Our approach builds on well-established techniques developed in the context of numerically verified proofs ---see, for example, \cite[Section 2.3]{GomezSerrano2019} and the references therein---and employs a subdivision strategy of a bounding box in the ambient space containing $Z(f)$.
The bounding box is decomposed into sufficiently small boxes, classified according to whether they contain points of $M$ or yield explicit gradient lower bounds. More precisely, \cref{algorithm:lower-bound-for-df} in the article works by subdividing the bounding box into boxes small enough so that each of the boxes satisfies a sufficient (but not necessary) criterion for one of the following two things:
(a) the box contains no point of $M$ (called "CaseOneBoxes" below);
(b) on the box, $|\nabla f|_1$ is bounded from below (called "CaseTwoBoxes" below).
Given such a subdivision, one obtains a lower bound for $|\nabla f|_1$ on all of $M$. Our first main result guarantees the correctness of this subdivision-based algorithm (\cref{proposition:df-algorithm-correctness}):

\begin{proposition}
    \label{proposition:main-prop}
    Assume $Z(f) \subset [-M_1,M_1]^N$ and $|\nabla f|_2 \leq M_2$ and $|\Hess f|_2 \leq M_3$ on $[-M_1,M_1]^N$.
    If $Z(f)$ is non-singular, then \cref{algorithm:lower-bound-for-df} terminates after finitely many steps, and yields a set $B \subset \R^N$ with $Z(f) \subset B$ such that
    \begin{align*}
        \left| \nabla f|_{B} \right|_1 > M_3 N^{3/2} \frac{\epsilon}{2}
    \end{align*}
    as well as
    \begin{align*}
        \left|
        f|_{[-M_1,M_1]^N \setminus B}
        \right|
        >
        \sqrt{N} \frac{\epsilon}{2} M_2,
    \end{align*}
    where $\epsilon$ is the side length of the smallest box in CaseOneBoxes$\,\cup\,$CaseTwoBoxes.
    Here $(v_1,\dots,v_N)|_1:=|v_1|+\dots+|v_N|$ denotes the $1$-norm of the vector $(v_1,\dots,v_N)$.
\end{proposition}

We then explain how such a bound on the gradient of $f$ can be used to obtain a lower bound for the reach of $M$.
Roughly speaking, if $\nabla f$ changes its direction quickly along $M$, then $M$ has large curvature.
If one has a lower bound for $\nabla f$ by virtue of the previous proposition, and an upper bound for $\Hess f$ (which is easy to come by), then that together means that $\nabla f$ cannot change direction arbitrarily fast.
In practice, one can turn this into a useful bound for the reach of $M$, which is the main use of the above proposition.
This is the following result (\cref{corollary:lower-bound-for-reach}):

\begin{corollary}
    Let $M=Z(f) \subset U$ for a convex set $U \subset \R^N$. 
    Let $|{\Hess f}|_2
        \leq
        \Cl{one-func-Hess-bound}$ on $U$ and $ |\nabla f|_1
        \geq
        \Cl{one-func-grad-bound}$ on $M$,
    then
    \begin{align*}
        \tau
        \geq
        \frac{C_2}{\sqrt{N} C_1}.
    \end{align*}
\end{corollary}

All of this generalises to submanifolds of higher codimension, which are defined as the vanishing sets of multiple smooth functions.
This requires some calculations in linear algebra, and we have moved the explanation of this generalisation to \cref{section:generalise-to-f1-...-fk}.

The availability of a computable lower bound for the reach has numerous applications across geometry, topology, and computational mathematics. In particular, we highlight:

\begin{enumerate}
    \item Extrinsic versus intrinsic distances (\cref{application:extrinsic-instric-distance-estimate}):
    Improved estimates relating intrinsic manifold distances to ambient-space distances in $\R^N$. This may have future applications in numerically verified proofs.

    \item Homology groups (\cref{section:application-computation-of-homology}):
    As explained above, one of the main problems in the research area is the computation of homology groups of a variety from finite data.
    A lower bound for the reach can be used as input for existing homology computation algorithms, but in this section, we describe an alternative approach that does not use sampling.
    Our proof works only in dimension $2$, and we leave the general case for future work.

    \item Eigenvalue estimate for the Laplacian 
    \Cref{section:application-lowest-eigenvalue-bound}:
    the lowest non-zero eigenvalue of the Laplace operator is a quantity of interest in geometric analysis.
    We explain how a lower bound for the reach immediately implies a lower bound for the lowest non-zero eigenvalue of the Laplace operator.
    Because we are using a naive estimate for the diameter of a submanifold in one step, our lower bound is very far from being sharp, and it is in fact so small that it is not usable in computations.
    There is hope this can be improved in the future, but for now, the result is still of theoretical interest.

    \item Stable deformation of algebraic varieties
    \Cref{section:application-deforming-algebraic-varieties}:
    given a smooth variety, small deformations of it remain smooth.
    Using the output of \cref{proposition:main-prop} (which is a bit more than just the reach), we immediately obtain an effective bound which guarantees that deformations that are smaller than the given bound are smooth.
    This is important in applications in pure mathematics, where one sometimes wants to deform a variety without changing its diffeomorphism type.
\end{enumerate}

The paper is structured as follows:
in \cref{section:background} we define the reach of a manifold and set up some other notation.
In \cref{section:lower-bound-for-df} we explain the algorithm for obtaining a lower bound for $|\nabla f|_1$ and prove its correctness.
We then use this bound to obtain a lower bound for the reach in \cref{section:lower-bound-for-reach}.
In \cref{application:extrinsic-instric-distance-estimate,section:application-computation-of-homology,section:application-lowest-eigenvalue-bound,section:application-deforming-algebraic-varieties} we marshal the applications listed above.
In \cref{section:tubular-nbhd} we give a proof of the quantitative tubular neighbourhood theorem, which is a folklore result for which we could not find a proof in the literature.

\textbf{Acknowledgments.}
The authors thank Oliver Gäfvert for helpful conversations.
D.P. is supported by the Eric and Wendy Schmidt AI in Science Postdoctoral Fellowship, funded by Schmidt Sciences.

\section{Background}
\label{section:background}

Throughout the paper, $M$ will be a smooth, closed embedded submanifold of the Euclidean space $\R^N$. 
We are particularly interested in the case where $M$ is given as the zero set of a finite collection of smooth functions, say $M=Z(f_1,\dots,f_k)$.
The special case that $M$ is a smooth real algebraic variety, i.e. the $f_i$ are real polynomials, is typically of most interest, though our methods apply more generally to smooth functions.

\begin{definition}
    \label{definition:distance-functions}
    For $x=(x_1,\dots,x_N),y=(y_1,\dots,y_N) \in \R^N$ we write
    \begin{align*}
        |x-y|_{\R^N}
        =
        \sqrt{(x_1-y_1)^2 + \dots + (x_N-y_N)^2}
    \end{align*}
    for its Euclidean distance.
    Given a point $x\in \R^n$, and a subset $S \subset \R^N$ the distance from $x$ to $S$ is defined as
    \begin{align*}
       d(x,S) := \inf_{s\in S}  |s-x|_{\R^N}.
    \end{align*}
    If $x,y \in M$, then we also have the geodesic distance on $M$ defined by the induced Riemannian metric, which we denote by $d_M(x,y)$.
\end{definition}

\begin{definition}[\cite{Federer1959}]
    For $p \in \R^N$, let $\pi_M(p)$ be the set of points $ x$ in $M$ such that $d(p,M) = | p- x |_{\R^N}$. 
    Let $U^r_M$ be the set of points in $\R^N$ whose distance to $M$ is less than $r$,
    \begin{align*}
     U^r_M = \{p \in \R^n: d(p,M)< r \}.  
    \end{align*}
    The \emph{reach} of $M$ in $\R^N$, denoted $\tau$ is
\begin{align}
    \label{equation:reach-def}
    \tau  = \sup \{ r \geq 0 : |\pi_M(p)|=1, \, \forall p \in U^r_M\}.
\end{align}
\end{definition}

Let $NM=\{ (x,v): x \in M \subset{\R^N}, v \in T_xM^\perp \}$ denote the normal bundle of $M$.
Throughout the paper, we will view $T_xM^\perp$ as a subspace of $\R^N$.
We then define:
\begin{align*}
 V^r = \{(x,v) \in NM: |v|_{\R^N} < r \} \subset NM.   
\end{align*}

The tubular neighbourhood theorem states that $M$ has a neighbourhood inside $\R^N$, which looks like a neighbourhood of the zero section within the normal bundle of $M$.
In fact, the endpoint map,
\begin{align*}
    \begin{split}
    E: V^\tau & \rightarrow U^\tau_M \\
      (x,v)      &\mapsto x+v,
    \end{split}
\end{align*}
is a diffeomorphism (\cref{theorem:tubular-nbhd-theorem}).
Furthermore, $r=\tau$ is the largest value of $r$ so that $E: V^r \rightarrow U^r_M$ is injective.

The following useful result follows easily:

\begin{proposition}
    \label{proposition:small-ball-has-connected-intersection-with-M}
    Let $M=Z(f) \subset \R^N$ and $\tau$ be the reach of $M$.
    Let $B(r)$ be an open ball of radius $r \leq \tau$ centred around an arbitrary point in $\R^N$.
    Then $M \cap B(r)$ has at most one connected component.
\end{proposition}
\begin{proof}
Assume the intersection $M \cap B(r)$ has two connected components $C_1$ and $C_2$.
Denote the centre of $B(r)$ by $x$ and let $c_i \in C_i$ be a closest point to $x$.
We see as in \cref{theorem:tubular-nbhd-theorem} that $c_i-x \perp T_{c_i} M$.
Thus, $x=E(c_1,c_1-x)=E(c_2,c_2-x)$ which contradicts the injectivity of the endpoint map $E$.
\end{proof}

The reach is related to Riemannian geometry through the notion of principal curvature:

\begin{definition}
    \label{definition:second-fundamental-form}
    The \emph{second fundamental form} is the operator
    \begin{align}
    \begin{split}
        \operatorname{II}: \mathfrak{X}(M) \times \mathfrak{X}(M) &\rightarrow \mathfrak{X}(M)^\perp
        \\
        (X,Y) & \mapsto \operatorname{nor}  \nabla_X Y,
        \end{split}
    \end{align}
    where $\nabla$ is the Levi-Civita connection on $\R^N$ and $\operatorname{nor}$ denotes the projection onto the normal bundle of $M$.

    Let $n$ be a normal vector perpendicular to $M$ at $x$.
    If we assume that the local coordinates are chosen so that the induced metric at $x \in M$ is the identity, then the eigenvalues of the symmetric matrix $\langle \operatorname{II}(\partial_i, \partial_j),n \rangle $ are called the \emph{principal curvatures} of $M$ at $x$ in the direction $n$.
    The inverses of the non-zero principal curvatures are called \emph{principal radii of curvature}.
\end{definition}

\begin{definition}
    For $p \in M$ let
    \begin{align*}
        \rho(p)
        :=
        \left(        
        \max_{\substack{\gamma \text{ unit speed} \\ \text{geodesic with } \gamma(0)=p}}
        |\gamma''(0)|_{\R^N}
        \right)^{-1} = \left(
\max_{\substack{v \in T_p M, |{v}|_{\R^N}=1}}
        |{\operatorname{II}(v,v)}|_{\R^N}
\right)^{-1}.
    \end{align*}
The equality follows from the fact that the acceleration of a geodesic in $M$ is always normal to $M$ (see \cite[Corollary 4.9]{ONeill1983}). 
The quantity 
\begin{align*}
    \rho:= \inf_{p \in M} \rho(p)
\end{align*}
is called the \emph{minimal radius of curvature}.
\end{definition}

In other words, $\rho(p)$ is the smallest principal radius of curvature on $p$ among all normal directions.

The critical points of the endpoint map are called \emph{focal points}. They represent nearby points in $M$ whose normal lines intersect.

\begin{proposition}\cite{milnor1963morse}
    The focal points along the straight line going through $x\in M \subset \R^N$ with direction $n \in T_xM^\perp$, are precisely the points $x + K_i^{-1} n$, where $K_i$ are the nonzero principal curvatures of $M$ at $x$ in the normal direction $n$.
\end{proposition}

At the focal points, the endpoint map $E$ fails to be injective; hence, the reach is bounded from above by the minimal radius of curvature.  
In order to give a precise relation between the reach and the minimal radius of curvature, we need to introduce the notion of bottleneck.

\begin{definition}
    The \emph{bottleneck locus} of $M$, denoted $B_2$, consists of the pairs $(x_1,x_2) \in M \times M$ with $x_1\neq x_2$ such that the straight line joining $x_1$ and $x_2$ in $\R^N$ is normal to $M$ at $x_1$ and at $x_2$. 
    Let $\lambda$ be the minimum of the distances between the points forming a pair in the bottleneck locus,
    \begin{align*}
        \lambda = \min_{(x_1,x_2)\in B_2} |x_1 - x_2|_{\R^N}.
    \end{align*}
    We will call $\lambda$ \emph{the length of the
smallest bottleneck} on $M$ and the quantity 
    \begin{align*}
        \eta := \frac{\lambda}{2}
    \end{align*}
      the \emph{radius of the smallest bottleneck} of $M$.
\end{definition}

\begin{theorem}[Theorem 3.4 \cite{Aamari2019}] 
\label{theorem:reach-characterisation}
Let $M$ be a compact smooth manifold embedded in $\R^N$, let $\eta$ be the radius of the smallest bottleneck of $M$ and $\rho$ the minimal radius of curvature, then
\begin{align*}
    \tau = \min \{ \eta, \rho \}.
\end{align*}  
\end{theorem}

\section{A lower bound for \texorpdfstring{$|\nabla f|_1$}{|\nabla f|_1}  
 on \texorpdfstring{$Z(f)$}{Z(f)}}
\label{section:lower-bound-for-df}

Assume that $M=Z(f) \subset \R^N$ is a closed submanifold defined as the zero set of a single smooth function.
This section's results generalise to zero sets of more than one function, as explained in \cref{section:generalise-to-f1-...-fk}.
If $\nabla f(x)=0$ for a point $x \in M$, then $x$ is a singular point.
Though not rigorously defined, one may informally think of $M$ as having infinite curvature at $x$, and therefore by \cref{theorem:reach-characterisation} as having reach $\tau=0$.
Thus, if we are interested in a lower bound for $\tau$ on a smooth manifold, having a lower bound for 
\begin{align*}
    |\nabla f(x)|_1
        :=
        \sum_{i=1}^N
        |\partial_i f (x)|
\end{align*}
on $Z(f)$ will be helpful.
Given a smooth variety $M$, in this section, we present an algorithm which computes such a lower bound for $|\nabla f|_1$.

The algorithm's idea is to decompose a bounded set containing $M$ into many small boxes.
For a single small box $b$ there is an easy-to-check necessary criterion for it to contain a point of $M$.
Namely, if $f$ is very large in the midpoint of $b$ compared to an upper bound of $|\nabla f|_{\R^N}$, then $f$ cannot vanish anywhere on $b$, and $M \cap b=\emptyset$.
Thus, if we are only interested in a lower bound on $M$, then we can disregard $b$.

On the other hand, we may be presented with a small box $\tilde{b}$ that does not meet this necessary criterion.
We need not know for certain whether $\tilde{b}$ contains a point in $M$ or not, so we compute a lower bound for $|\nabla f|_1$ on all of $\tilde{b}$.
We obtain this lower bound in a similar way:
if $|\nabla f|_1$ is large in the midpoint of $\tilde{b}$ compared to an upper bound of $|\Hess f|_2$, that gives a lower bound for $|\nabla f|_1$ on all of $\tilde{b}$.

This idea is a standard technique in numerically verified proofs; see, for example, \cite[Section B]{Buckmaster2022} and \cite[Section 4.3.2]{GomezSerrano2014}.

This idea is made precise in \cref{algorithm:lower-bound-for-df}, and \cref{figure:algorithm-step-by-step} shows the algorithm being applied, step-by-step, to the function $f(x,y)=x^2+y^2-1$.

\begin{algorithm}
  \caption{Finding a lower bound for $\nabla f$ on $Z(f)$. For an $n$-dimensional box CurrentBox, the function 
  $\text{subdivide}(\text{CurrentBox})$ returns $2^n$ boxes of half side length whose union is CurrentBox.
  We write $|A|_2=\max_{x \neq 0} \frac{|Ax|_{\R^N}}{|x|_{\R^N}}$ for the $2$-norm of the matrix $A$.
    \label{algorithm:lower-bound-for-df}}
  \begin{algorithmic}[1]
    \Require{$M_1,M_2,M_3$ such that $Z(f) \subset [-M_1,M_1]^N$ and $|\nabla f|_{\R^N} \leq M_2$ and $|\Hess f|_2 \leq M_3$ on $[-M_1,M_1]^N$}
    \Statex
    \Let{NewBoxes}{$\{[-M_1,M_1]^N\}$}
    \Let{CaseOneBoxes}{$\emptyset$}
    \Let{CaseTwoBoxes}{$\emptyset$}

    \While{NewBoxes $\neq \emptyset$}
    \Let{CurrentBox}{pop(NewBoxes)}
    \Let{$m$}{midpoint(CurrentBox)}
    \Let{$\varepsilon$}{sidelength(CurrentBox)}
    \If{$|f(m)|> \sqrt{N} \epsilon M_2$}
        \Let{CaseOneBoxes}{CaseOneBoxes $\cup \{\text{CurrentBox} \}$ }
    \ElsIf{$|\nabla f(m)|_1> N^{3/2} \epsilon M_3 $}
        \Let{CaseTwoBoxes}{CaseTwoBoxes $\cup \{\text{CurrentBox} \}$ }
    \Else
        \Let{NewBoxes}{NewBoxes $\cup \text{subdivide}(\text{CurrentBox})$}
        \EndIf
    \EndWhile
  \end{algorithmic}
\end{algorithm}

\begin{figure}[!htbp]
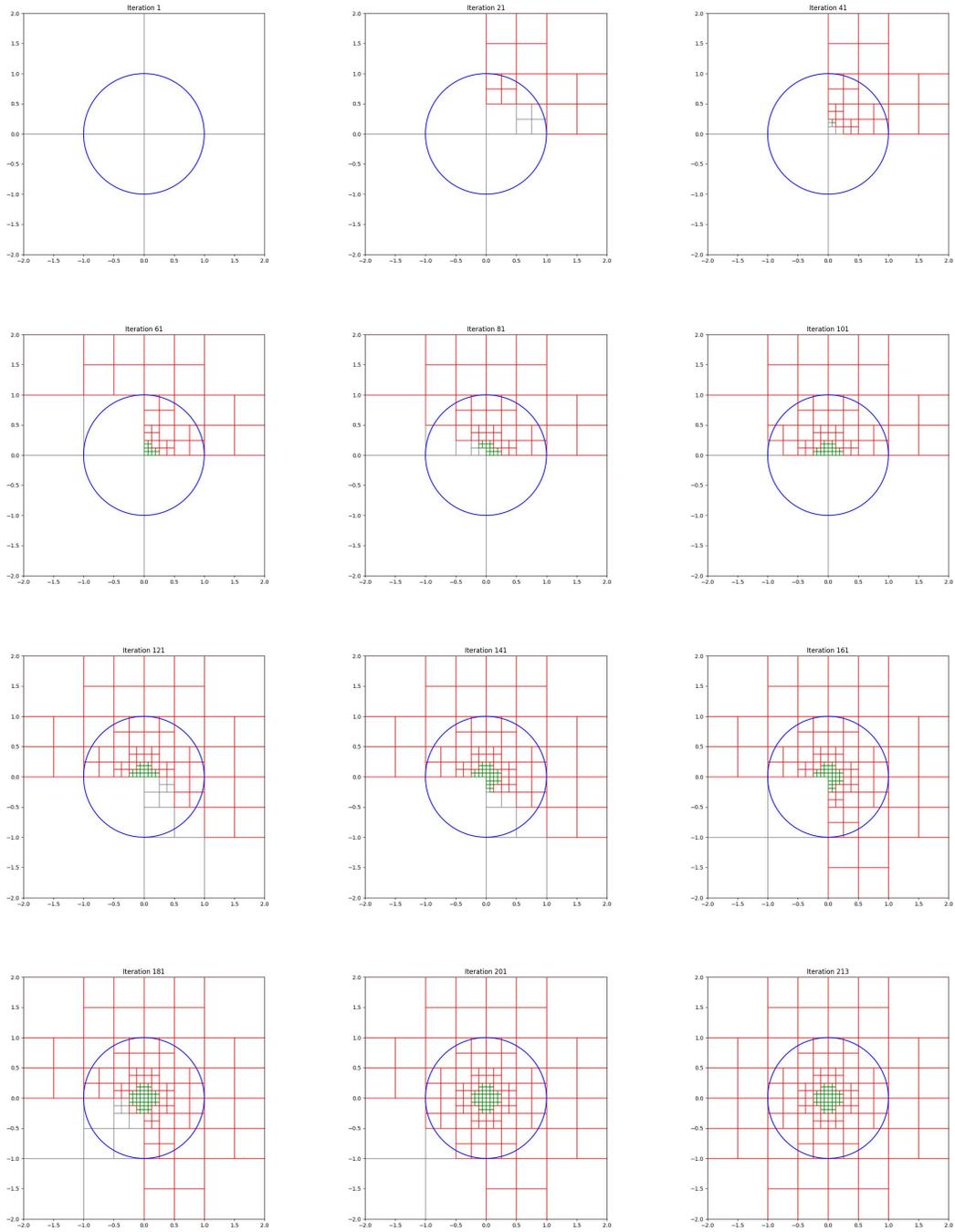

\centering

\newcommand{\imgnum}[1]{\includegraphics[width=0.3\linewidth]{img/algo_step_0#1.png.pdf.ps}}

\begin{tabular}{ccc}
\imgnum{001} & \imgnum{021} & \imgnum{041} \\
\imgnum{061} & \imgnum{081} & \imgnum{101} \\
\imgnum{121} & \imgnum{141} & \imgnum{161} \\
\imgnum{181} & \imgnum{201} & \imgnum{213} \\
\end{tabular}

\caption{\Cref{algorithm:lower-bound-for-df} being executed for the function $f(x,y)=x^2+y^2-1$ on $[-M_1,M_1]$ for $M_1=2$. Boxes in NewBoxes are drawn with a gray border, boxes in CaseOneBoxes are drawn with a green border, and boxes in CaseTwoBoxes are drawn with a red border. The algorithm takes $213$ steps to terminate, and we show progress at 12 steps throughout the process.}
\label{figure:algorithm-step-by-step}
\end{figure}

\begin{proposition}
    \label{proposition:df-algorithm-correctness}
    If $Z(f)$ is non-singular, then \cref{algorithm:lower-bound-for-df} terminates after finitely many steps, and yields a set $B \subset \R^N$ with $Z(f) \subset B$ such that
    \begin{align}
        \label{equation:df-lower-bound}
        \left| \nabla f|_{B} \right|_1 > M_3 N^{3/2} \frac{\epsilon}{2}
    \end{align}
    as well as
    \begin{align}
        \label{equation:f-lower-bound}
        \left|
        f|_{[-M_1,M_1]^N \setminus B}
        \right|
        >
        \sqrt{N} \frac{\epsilon}{2} M_2,
    \end{align}
    where $\epsilon$ is the side length of the smallest box in CaseOneBoxes$\,\cup\,$CaseTwoBoxes.
    Here $(v_1,\dots,v_N)|_1:=|v_1|+\dots+|v_N|$ denotes the $1$-norm of the vector $(v_1,\dots,v_N)$.
\end{proposition}

\begin{proof}
    We first show that \cref{algorithm:lower-bound-for-df} terminates after finitely many steps.
    To this end, assume that the algorithm does not terminate.
    That is, there exists an infinite sequence of boxes $b_0=[-M_1,M_1]^N \supset b_1 \supset b_2 \supset \dots$ with the property that $b_{i+1}$ has half the side length of $b_i$ for all $i \in \{0, 1, 2, \dots\}$, and each $b_i$ is in the "if" or "else if" case from \cref{algorithm:lower-bound-for-df}, but never the "else" case.
    For each $i$, let $p_i \in b_i$.
    By assumption on the boxes, we have that the limit $\lim_{i \rightarrow \infty} p_i=p^*$ exists.
    Because $b_i$ is not in the "if" case, we have that
    \begin{align*}
            |f(p^*)|
        =\lim_{i \rightarrow \infty} |f(p_i)| \leq \lim_{i \rightarrow \infty} \frac{2M_1}{2^i} M_2
        =0.
    \end{align*}

    Because $b_i$ is not in the "else if" case, we have that
    \begin{align*}
        |\nabla f (p^*)|_1
        =
        \lim_{i \rightarrow \infty} |\nabla f(p_i)|_1
        \leq
        \lim_{i \rightarrow \infty} \frac{1}{2} \cdot \frac{2M_1}{2^i} M_3
        =0.
    \end{align*}
    This means that $p^*$ is a singular point of $Z(f)$, which contradicts the assumption that $Z(f)$ is non-singular.

    Define $B:= \bigcup_{b \in \text{CaseOneBoxes}} b$.
    It remains to show the bounds of \cref{equation:df-lower-bound,equation:f-lower-bound}.
    When \cref{algorithm:lower-bound-for-df} terminates, then NewBoxes$\,=\emptyset$.
    
    Let $b \in $ CaseOneBoxes be a box with side length $\epsilon$ and midpoint $m \in b$ satisfying the "if" case.
    For $x \in b$ let $\gamma:[0,L] \rightarrow \R^N$ be the unit speed straight line from $m$ to $x$.
    By the mean value theorem, there exists $\theta \in (0,L)$ such that
    \begin{align}
    \label{equation:algo-f-nonzero-argument}
    \begin{split}
        |f(x)|
        &=
        \left| f(m)- \< \nabla f (\gamma(\theta)), \gamma'(\theta) \> \cdot |m-x|_{\R^N}\right|
        \\
        &
        \geq
        |f(m)|- \sqrt{N} \frac{\epsilon}{2} M_2
        \\
        &
        >
        \sqrt{N} \frac{\epsilon}{2} M_2
        \\
        &
        >0,
    \end{split}
    \end{align}
    where we used that the maximum distance of $m$ and $x$ in the $N$-dimensional cube of side length $\epsilon$ is $\sqrt{N} \frac{\epsilon}{2}$, and we used the assumption from the "if" case in the second step.
    Thus, $Z(f) \cap b = \emptyset$.

    Now, assume $b$ satisfies the "only if" case.
    Similar to the previous case, by the mean value theorem, there exists for all $i \in \{1,\dots,n\}$ and $x \in b$ some $\theta \in (0,L)$ such that,
    \begin{align*}
        |\partial_i f(x)|
        =
        \left| \partial_i f(m)- \< \nabla \partial_i f (\gamma(\theta)), \gamma'(\theta)\> \cdot |m-x|_{\R^N}\right|,
    \end{align*}
    and therefore
    \begin{align}
    \label{equation:which-Hess-norm-used}
    \begin{split}
        |\nabla f(x)|_1
        &=
        \sum_{i=1}^N
        |\partial_i f (x)|
        \\
        &\geq
        |\nabla f(m)|_1
        -
        \sum_{i=1}^N
        |\nabla \partial_i f(\gamma(\theta))|_{\R^N}
        |m-x|_{\R^N}
        \\
        &=
        |\nabla f(m)|_1
        -
        \sum_{i=1}^N
        |\Hess_i f(\gamma(\theta))|_{\R^N}
        |m-x|_{\R^N}
        \\
        &\geq
        |\nabla f(m)|_1
        -
        M_3 N^{3/2} \frac{\epsilon}{2}
        \\
        &>
        M_3 N^{3/2} \frac{\epsilon}{2},
        \end{split}
    \end{align}
 where $\Hess_i f(\gamma(\theta))$ is the $i$-th row of $\Hess f(\gamma(\theta)$, and we used the assumption from the "else if" condition in the last step.
    This gives a lower bound for $|\nabla f|_1$ on $b$, and therefore in particular on all of $Z(f)$, because $Z(f)$ is contained in the union of all boxes satisfying the "else if" condition.
\end{proof}

\begin{remark}
    \label{remark:possible-algo-improvements}
    We presented the most basic version of \cref{algorithm:lower-bound-for-df}, but several improvements are possible to improve its efficiency.
    \begin{enumerate}
        \item 
        In \cref{equation:which-Hess-norm-used} we obtain a better constant if using the $(2,1)$-norm (see \cite[Section 2.3.1]{Golub2013} for a definition) of $\Hess f$ rather than the $2$-norm.
        However, this norm may be difficult to compute explicitly.

        \item 
        We subdivide boxes into $2^N$ smaller boxes.
        A standard improvement, for example, implemented in \cite{GomezSerrano2019}, is to only subdivide along one direction, thereby producing two rectangles.
        Roughly speaking:
        if one wants to ensure that $\nabla f$ is not vanishing on a box, this becomes easier when the length of the box in the direction of $\nabla f$ is small.
        Thus, it can speed up the algorithm to subdivide in this direction.

        \item 
        Instead of using bounds $|\nabla f|_{\R^N} \leq M_2$ and $|\Hess f|_2 \leq M_3$ that hold \emph{everywhere}, a box is more likely to fall into case one or case two from the algorithm if one uses bounds for $\nabla f$ and $\Hess f$ on the box that is currently being checked.
        One also gets a sharper estimate for $|\nabla f|_1$ if instead of using the global bounds in \cref{equation:df-lower-bound} one evaluates the right hand side of \cref{equation:df-lower-bound} on every box in CaseTwoBoxes individually, with the local bounds for $\nabla f$ and $\Hess f$.
    \end{enumerate}
\end{remark}

\begin{example}
    \label{example:circle}
    For $f(x,y)=x^2+y^2-1$, running \cref{algorithm:lower-bound-for-df} on $M:= Z(f) \subset [-2,2]^2$ requires $213$ steps and yields $|\nabla f|_1 \geq 0.3535$ on $M$.
    Some intermediate steps of running the algorithm are shown in \cref{figure:algorithm-step-by-step}.
    The obtained bound is correct but not sharp; the global minimum is:
   \begin{align*}
    \min_{x \in M} | \nabla f|_1(x)
    =
    |\nabla f|_1
    \left(
    1,0
    \right)
    =
    2.
   \end{align*}
    One can obtain sharper bounds by choosing a smaller square than $[-2,2]^2$ containing $M$, which leads to better bounds for $\grad f$ and $\Hess f$.
\end{example}

\begin{example}
    \label{example:interesting-curve}
    We also apply \cref{algorithm:lower-bound-for-df} to the curve
    \begin{align*}
    M=Z(f) \subset \R^2 \quad \text{ where } \quad
    f(x,y)
    =
    (x^3 - x y ^ 2 + y + 1) ^ 2  (x ^ 2 + y^ 2 - 1) + y^ 2 - 5
    \end{align*}
    from \cite{Breiding2025} (see also \cite[Figure 10]{DiRocco2022}).
    As written, \cref{algorithm:lower-bound-for-df} requires a long time to terminate, and we therefore implemented the third improvement from \cref{remark:possible-algo-improvements}.
    The result after $5729$ steps is $|\nabla f|_1 \geq 0.6496$.
    Randomly sampling $10000$ points in $M$ and computing $| \nabla f|_1$ in these points gives a minimum of $1.88689$, which is consistent with this estimate.
\end{example}

\section{A lower bound for the reach}
\label{section:lower-bound-for-reach}

As before, we denote by $M$ the smooth manifold given by the zero set $Z(f) \subset \R^N$ of a smooth function $f:\R^N \rightarrow \R$.
In \cref{section:lower-bound-for-df}, we provided an algorithm to compute a lower bound for $|\nabla f|_1$ on $M$.
We now put this lower bound to work to prove a lower bound for the reach of $M$.
By virtue of \cref{theorem:reach-characterisation}, this requires two separate estimates:
first, an upper bound for the curvature of $M$, which is achieved in \cref{subsection-upper-bound-for-curvature-from-df} and second, a lower bound for the smallest bottleneck of $M$, which is given in \cref{subsection:lower-bound-for-bottleneck}.
For convenience, we explicitly state the bound for the reach obtained in this way in corollary \ref{corollary:lower-bound-for-reach}.

\subsection{An upper bound for the curvature of $M$}
\label{subsection-upper-bound-for-curvature-from-df}

We recall the heuristic picture that at a singular point in $M$, i.e. $\nabla f=0$, we have that the curvature of $M$ "is infinity".
In more rigorous terms, if $\nabla f$ is close to zero, then the curvature of $M$ is very large.
The next proposition is a standard exercise in Riemannian geometry, and exactly quantifies this heuristic:

\begin{proposition}[Exercise 4.3 in \cite{ONeill1983}]
    \label{proposition:oneill-II-exercise}
    Let $f:\R^N \rightarrow \R$ be a smooth function having $0$ as a regular value. If $M=Z(f) \subset \R^N$, then, for unit length tangent vectors $v,w \in T_x M$ we have
    \begin{align}
        |\operatorname{II}(v,w)|_{\R^N}
        =
        \frac{| \Hess(f)(v,w) |}{| \grad f(x)|_{\R^N}}.
    \end{align}
\end{proposition}

\subsection{A lower bound for the smallest bottleneck of $M$}
\label{subsection:lower-bound-for-bottleneck}

The reach is controlled by two quantities:
the second fundamental form and the smallest bottleneck.
We obtained an estimate for the second fundamental form in the last \cref{subsection-upper-bound-for-curvature-from-df}.
Thus, it remains to prove an estimate for the smallest bottleneck, which is precisely the following:

\begin{proposition}
    \label{proposition:bottle-neck-bound-one-function}
    Let $M=Z(f) \subset U$ for some convex set $U$ in $\R^N$.
    Assume $|\nabla f|_{\R^N} \geq \epsilon$ on $Z(f)$ and assume $|\Hess f|_{2} \leq k$ on $U$.
    Here, $|A|_2=\max_{x \neq 0} \frac{|Ax|_{\R^N}}{|x|_{\R^N}}$ for a matrix $A$ as before.
    Denote by $\lambda$ the length of the smallest bottleneck of $M$.
    Then:
    \begin{align}
        \lambda
        \geq
        \frac{\epsilon}{k}.
    \end{align}
\end{proposition}

\begin{proof}
    Let $p,q \in M$ be two points connected by a smallest bottleneck, i.e. $|p-q|_{\R^N}=\lambda$.
    Denote by $\gamma : [0,\lambda] \rightarrow U$ the unit speed straight line from $p$ to $q$.
    We have that $f \circ \gamma:[0,\lambda] \rightarrow \R$ vanishes at the endpoints, so by Rolle's theorem, there exists $\theta \in (0, \lambda)$ such that 
    \begin{align}
        \label{equation:theta-property}
        0=\frac{d}{dt}|_{t=\theta} f(\gamma(t))=\< \nabla f (\gamma(\theta)),\gamma'(\theta) \>.
    \end{align}
    By the mean value theorem, there exists $\xi \in (0,\theta)$ such that
    \begin{align}
        \label{equation:xi-mean-value-property}
        \left|
        \frac{d}{dt}|_{t=\theta}
        f(\gamma(t))
        -
        \frac{d}{dt}|_{t=0}
        f(\gamma(t))
        \right|
        =
        \left|
        \theta
        \cdot
        \frac{d^2}{dt^2}|_{t=\xi}
        f(\gamma(t))
        \right|.
    \end{align}

    Thus
    \begin{align*}
    \begin{split}
        \epsilon 
        &\leq
        | (\nabla f) (\gamma(0))|_{\R^N}
        =
        | (\nabla f) (\gamma(0))|_{\R^N} \cdot |\gamma'(0)|_{\R^N} =
        |\< (\nabla f) (\gamma(0)), \gamma'(0) \>|
        \\
        &=
        \left|
        \frac{d}{dt}|_{t=0}
        f(\gamma(t))
        \right|
        =
        \left|
        \frac{d}{dt}|_{t=\theta}
        f(\gamma(t))
        -
        \frac{d}{dt}|_{t=0}
        f(\gamma(t))
        \right|
        \\
        &=
        \left|
        \theta
        \cdot
        \frac{d^2}{dt^2}|_{t=\xi}
        f(\gamma(t))
        \right|
        =
        \theta
        \left|
        \frac{d}{dt}|_{t=\xi}
        \< \nabla f, \gamma'(t) \>
        \right|
        \\
        &=
        \theta
        \left|
        \<\nabla_{\gamma'(\xi)} \nabla f, \gamma'(\xi) \>
        +
        \<\nabla f, \nabla_{\gamma'(t)} \gamma'(t) \>
        \right|
        \\
        &\leq
        \theta
        | \nabla_{\gamma'(\xi)} \nabla f|_{\R^N} 
        \cdot 
        | \gamma'(\xi)|_{\R^N} 
        \leq
        \theta
        |\Hess f(\gamma(\xi))|_{2} \cdot |\gamma'(\xi)|_{\R^N}
        \\
        &\leq
        \lambda k,
        \end{split}
    \end{align*}
    proving the claim.
    In the second step, we used that $\gamma$ is a unit speed curve;
    In the third step, we used the definition of bottleneck, which implies that $\gamma'(0)$ is perpendicular to $Z(f)$, together with the fact that $\nabla f$ is perpendicular to $Z(f)$;
    in the fifth step we used \cref{equation:theta-property};
    in the sixth step we used \cref{equation:xi-mean-value-property};
    in the eighth step, we used that the connection $\nabla$ is metric;
    in the ninth step, we used the Cauchy-Schwarz inequality for the first summand and $\nabla_{\gamma'(t)}\gamma'(t)=0$, because $\gamma$ is a geodesic with respect to $\nabla$;
    in the last step, we used again that $\gamma$ is a unit speed curve.
\end{proof}

\begin{remark}
    \label{remark:algo-comparison}
    We now comment on how our algorithm \cref{algorithm:lower-bound-for-df} compares to the approach of \cite{DiRocco2022} for computing the smallest bottleneck and \cite{2023DiRocco} for computing a lower bound for the reach:

    it is easy to construct examples that are immediately solved by \cite{DiRocco2022,2023DiRocco}, but can take a very long time for \cref{algorithm:lower-bound-for-df}.
    For example, taking the union of two circles that are close together, they form a very small bottleneck, say
    \[
    f(x,y)
    =
    (x^2+y^2-1)((x-\xi)^2+y^2-1).
    \]
    For $\xi>2$, this has a bottleneck of length $\xi - 2$.
    The algorithms of \cite{DiRocco2022,2023DiRocco} computes this for any choice of $\xi$ immediately, while \cref{algorithm:lower-bound-for-df} needs a long time for small values of $\xi$.

    Conversely (and this was one of the primary motivations for this article), one can construct examples where the algorithms of \cite{DiRocco2022,2023DiRocco} do not terminate, but \cref{algorithm:lower-bound-for-df} still produces a result.
    For example, if $f(x,y)=x^2+y^2-1+\epsilon$, where $\epsilon$ is a small polynomial, then $Z(f)$ remains roughly a circle, and its reach does not change much.
    \Cref{algorithm:lower-bound-for-df} works the same as before.
    However, if $\epsilon$ is a generic polynomial of high degree, then the other algorithms take a long time.

    Apart from these artificial examples:
    In one example from \cite{Douglas2024}, the implementation \cite{Gäfvert_2020} of the algorithm from \cite{DiRocco2022} did not yield a result because no suitable starting system could be found for the homotopy method.
    \Cref{algorithm:lower-bound-for-df} (or more accurately: \cref{algorithm:lower-bound-for-det}) works, but we have not yet implemented parallelization which would be needed to make it finish in a reasonable time.
\end{remark}

\subsection{Deducing the lower bound for the reach}

Putting the results of this section together, we obtain the following corollary as an immediate consequence of \cref{theorem:reach-characterisation,proposition:bottle-neck-bound-one-function,proposition:oneill-II-exercise}:

\begin{corollary}
    \label{corollary:lower-bound-for-reach}
    Let $M=Z(f) \subset U$ for some convex set $U\subset \R^N$.
    If  $|{\Hess f}|_2
        \leq
        C_1
        $
        on $U$ and
        $ |\nabla f|_1 \geq C_2$ on M,
    then
    \begin{align}
        \tau
        \geq
        \frac{C_2}{2 \sqrt{N} C_1}.
    \end{align}
\end{corollary}

\begin{proof}
    We have $| \nabla f |_1 \leq \sqrt{N} | \nabla f |_{\R^N}$.
    Using proposition \ref{proposition:oneill-II-exercise}, we have $\rho \geq \frac{C_2}{\sqrt{N} C_1}$ and using \cref{proposition:bottle-neck-bound-one-function} we get $\eta = \frac{1}{2} \lambda \geq \frac{C_2}{2\sqrt{N} C_1}$.
    Plugging this into \cref{theorem:reach-characterisation} gives the claim.
\end{proof}

\begin{example}
    \label{example:circle-reach}
    Returning to the example of the circle from \cref{example:circle}:
    using the bounds
    \[
    |{\Hess f}|_2
    =
    \left|{ \begin{pmatrix} 2&0\\0&2 \end{pmatrix}}\right|_2
    =
    2
    \text{ on }
    [-2,2]^2
    \quad
    \text{ and }
    \quad
    |\nabla f |_1
    \geq
    0.3535
    \text{ on }
    M,
    \]
    we obtain the bound
    \[
    \tau
    \geq
    \frac{0.3535}{2 \cdot \sqrt{2} \cdot 2}
    \approx
    0.0625
    \quad
    \text{(true value: $\tau=1$).}
    \]
\end{example}

\begin{example}
    For the curve from \cref{example:interesting-curve}, we use the bounds
    \[
    |{\Hess f}|_2
    \leq
    22406.484
    \text{ on }
    [-3,3]^2
    \quad
    \text{ and }
    \quad
    | \nabla f |_1
    \geq
    1.55
    \text{ on }
    M
    \]
    and obtain
    \[
    \tau
    \geq
    1.03 \cdot 10^{-5}.
    \]
\end{example}

\section{Application: comparing extrinsic and intrinsic distance}
\label{application:extrinsic-instric-distance-estimate}

We now turn to applications of lower bounds for the reach, and begin with a comparison statement for the extrinsic and intrinsic distance for submanifolds $M \subset \R^N$.
Here it is no longer important that $M$ is the zero set of functions; we only require a lower bound for its reach.

Given a submanifold $M \subset \R^N$ recall the two distance functions on $M$ from \cref{definition:distance-functions}: 
the intrinsic distance $d_M$ and the extrinsic distance $|\cdot|_{\R^N}$ from Euclidean space.
Trivially, $|x-y|_{\R^N} \leq d_M(x,y)$.
The converse inequality 
\begin{align}
    \label{equation:extrinsic-intrinsics-estimate}
    d_M(x,y) \leq c \cdot |x-y|_{\R^N}
\end{align} for some constant $c>0$ is important in different applications.
In differential geometry, one often requires intrinsic diameter bounds, and one could obtain such an estimate from \cref{equation:extrinsic-intrinsics-estimate}.
Such bounds were, for example, studied in \cite[Theorem 1.1]{Topping2008}, \cite[Theorem 2.3]{Huisken1998}, \cite[Theorem 1]{Feng1999}.

Another application is the following:
a basic technique for obtaining an upper bound of a function in Euclidean space is the following \emph{Lipschitz method}.
One evaluates the function at many points, and together with a Lipschitz estimate, this yields a global estimate for the function.
This is the basic idea behind the more sophisticated methods in \cite{Katz2017,Eiras2023,Buckmaster2022}.
To use the same idea on submanifolds of $\R^N$, one can use an estimate like \cref{equation:extrinsic-intrinsics-estimate}.

We begin by stating one known result comparing extrinsic and intrinsic distances:

\begin{proposition}[Proposition 6.3 in \cite{Niyogi2008}]
    \label{proposition:intrinsic-extrinsic-sqrt-estimate}
    Let $M \subset \R^N$ be a submanifold with reach $\tau$
and $p,q \in M$. If $d:=|p-q|_{\R^N}\leq \frac{\tau}{2}$ then
    \begin{equation}
        d_M(p,q)
        \leq
        \tau-
        \tau \sqrt{1-\frac{2d}{\tau}}.
    \end{equation}
\end{proposition}

We rewrite this inequality slightly to make it look more similar to \cref{equation:extrinsic-intrinsics-estimate}:

\begin{corollary}
    \label{corollary:intrinsic-extrinsic-sqrt-estimate}
    In the situation of \cref{proposition:intrinsic-extrinsic-sqrt-estimate},
    \begin{equation}
        d_M(p,q)
        \leq
        \sqrt{2\tau} \cdot \sqrt{|p-q|_{\R^N}}.
    \end{equation}
\end{corollary}

\begin{proof}
    For $a \geq b \geq 0$ we have $2\sqrt{ab} \geq 2b$, which implies $a+b-2b \geq a+b-2\sqrt{ab}$ and therefore $\sqrt{a-b} \geq \sqrt{a}-\sqrt{b}$. Taking $a=1$ and $b=1- \frac{2d}{\tau}$ in this inequality we obtain,
    \[
        d_M(p,q)
        \leq
        \tau-
        \tau \sqrt{1-\frac{2d}{\tau}}
        =
        \tau
        \left(
        1-\sqrt{1-\frac{2d}{\tau}}
        \right)
        \leq
        \tau 
        \sqrt{\frac{2d}{\tau}}
        =
        \sqrt{2\tau} \cdot \sqrt{|p-q|_{\R^N}}. \qedhere
    \]
\end{proof}

We have not quite achieved the desired comparison inequality \cref{equation:extrinsic-intrinsics-estimate} yet:
in \cref{corollary:intrinsic-extrinsic-sqrt-estimate}, the distance on the right-hand side has a square root, while the left-hand side does not.
Thus, for small distances, \cref{corollary:intrinsic-extrinsic-sqrt-estimate} is weaker than our desired inequality.
However, it is reasonable to suggest that our desired inequality holds because every Riemannian manifold is locally Euclidean, so for small distances we see that $d_M$ is roughly equal to $|\cdot|_{\R^N}$.
We will make this rigorous in the remainder of this section, beginning with the following technical proposition:

\begin{proposition}
    \label{proposition:curve-with-distance-maximum-has-large-curvature}
    Let $\gamma:[0,L] \rightarrow \overline{B^N(\xi)}$ be a unit speed curve in the closure of the ball of radius $\xi$, centred at the origin (after a translation, we can assume it is centred in any point of $\R^N$).
    If there is $t_0 \in (0,L)$ such that $|\gamma(t)|_{\R^N}$ has a local maximum at $t_0$, then
    \begin{equation}
        |\gamma ''(t_0) |_{\R^N} \geq \frac{1}{\xi}.
    \end{equation}
\end{proposition}

\begin{proof}
    Let $u(t):=|\gamma(t)|^2_{\R^N}$.
    Then
    $u''(t)=2 \< \gamma''(t), \gamma(t)\>+2$.
    Thus, at $t_0$:
    \[
        0
        \geq
        u''(t_0)
        =
        2 \<\gamma''(t_0), \gamma(t_0)\> +2
        \geq
        -2 |\gamma''(t_0)|_{\R^N} \cdot |\gamma(t_0)|_{\R^N}+2,
    \]
    where we used the Cauchy-Schwarz inequality $|\< \gamma(t_0),\gamma''(t_0)\>| \leq |\gamma(t_0)|_{\R^N}\cdot |\gamma''(t_0)|_{\R^N}$ in the last step. 
    Hence
    \[
    |\gamma ''(t_0) |_{\R^N} \geq  \frac{1}{|\gamma(t_0)|_{\R^N}} \geq \frac{1}{\xi}.
    \qedhere
    \]
\end{proof}

Now comes another technical proposition.
Assume that $\gamma$ is a curve in a ball emanating from its midpoint.
If the curvature of the curve is very small compared to the size of the ball, it will leave the ball in a roughly straight line.
Thus, depending on the curvature of $\gamma$ and the radius of the ball, it is possible to prove an upper bound for the length of $\gamma$, which is achieved in the following proposition:

\begin{proposition}
    \label{proposition:trapped-small-curvature-curve-is-short}
    Let $\gamma:[0,L] \rightarrow \overline{B^N(\xi)}$ be a unit speed curve in the closure of the ball of radius $\xi$. If $\gamma(0)=0$, $|\gamma''| \leq \frac{\xi^{-1}}{2}$,  and $\gamma$ has a unique point of intersection with the boundary of the ball, which is given at $\gamma(L) \in S^{N-1}(\xi)$, then $L \leq 2 \xi$.
\end{proposition}

\begin{proof}
    Without loss of generality $\gamma'(0)=(1,0,\dots,0)$.
    We denote the first coordinate of $\gamma$ by $\gamma_1$ and 
    \begin{align*}
        |\gamma_1(t)|
        &=
        \left|
        \int_0^t \gamma_1'(s) \d s
        \right|
        =
        \left|
        \int_0^t \left( \gamma_1'(0)+ \int_0^s \gamma_1''(z) \d z \right) \d s
        \right| \geq   \left|
        \int_0^t  \gamma_1'(0) ds \right| - \left| \int_0^t \int_0^s \gamma_1''(z) \d z \d s
        \right| \\
        &\geq
        t-\frac{1}{2}t^2 \max(|\gamma''|)
   \end{align*}
    where we applied the fundamental theorem of calculus in the first two steps.
    Thus, at $t=2\xi$:
    \[
        |\gamma_1(2\xi)|
        \geq
        2\xi-2 \xi^2 \left( \frac{\xi^{-1}}{2} \right)
        =
        \xi.
    \]
    Thus, by time $2\xi$ the curve must intersect $S^{N-1}(\xi)$, so $L \leq 2\xi$.
\end{proof}

Putting everything together, we achieve our desired distance comparison inequality \cref{equation:extrinsic-intrinsics-estimate} with an explicit constant:

\begin{corollary}
    \label{corollary:intrinsic-extrinsic-distance-estimate}
    If $x,y \in M$ and $|x-y|_{\R^N} < \frac{\tau}{2}$, then
    \begin{equation}
        d_M(x,y) \leq 2 |x-y|_{\R^N}.
    \end{equation}
\end{corollary}

\begin{proof}
    By \cref{corollary:intrinsic-extrinsic-sqrt-estimate} we have $d_M(x,y) \leq \tau$.
    So, the shortest unit speed geodesic $\gamma:[0,L] \rightarrow M$ from $x$ to $y$ is contained in $\overline{B(x,\tau)}$.

    Write $\xi := |x-y|_{\R^N}$.
    Observe that the curve $\gamma |_{[0,L)}$ is contained in $B(x,\xi)$:
    if $\gamma$ leaves $B(x,\xi)$ before time $L$ but returns to $y \in \overline{B(x,\xi)}$ at time $L$, then it has a point of maximum distance from $x$ in $B(x,\tau)$.
    By \cref{proposition:curve-with-distance-maximum-has-large-curvature}, there exists $t_0$ such that $|\gamma ''(t_0)| > \frac{1}{\tau}$.
    This contradicts \cite[Proposition 6.1]{Niyogi2008}.

    Hence, $\gamma:[0,L] \rightarrow \overline{B^n(\xi)}$ is a unit speed curve with $\gamma(0)=x$, $\gamma(L)=y \in S^{n-1}(x,\xi)$ and
    \[
        |\gamma''|
        \leq
        \frac{1}{\tau}
        <
        \frac{1}{2\xi}.
    \]
    Moreover, $\gamma$ that meets $S^{n-1}(x,\xi)$ only at $\gamma(L)=y$.
    Thus, by \cref{proposition:trapped-small-curvature-curve-is-short}, $d_M(x,y)=L \leq 2 \xi$, which proves the claim.
\end{proof}

\section{Application: computation of homology groups}
\label{section:application-computation-of-homology}

Let $M=Z(f) \subset \R^N$ and as before, let $\tau$ be its reach.
One motivation for defining and computing the reach of $M$ is to compute its homology.
Roughly speaking, the reach tells how densely one needs to sample from $M$ so that one can recover the homology of $M$ from the persistent homology of the sample.
The introduction of \cite{Dufresne2019} gives an overview of work in the area.
\cite{Dufresne2019} produces arbitrarily dense samples, and if the density is high compared to $\tau$, then one recovers the homology of $M$.
However, it gives no way of bounding $\tau$. In
\cite{DiRocco2022} a guaranteed correct computation of the Betti numbers of $M$ is achieved using this strategy by bounding the smallest bottleneck, but without explicitly bounding $\tau$.

In \cref{section:lower-bound-for-reach} we explained how to obtain a lower bound for the reach.
This can immediately be used as input for the algorithms from \cite{Dufresne2019,DiRocco2022} in order to provably correctly compute the Betti numbers of $M$.

Apart from persistent homology, there are two more flavours for computing Betti numbers of algebraic varieties:
first, using techniques from semi-algebraic sets, pioneered in \cite{Basu2006}.
This yields provably correct results, but no implementation is publicly available to our knowledge.
Second, for some complex spaces there exist Groebner basis techniques for computing refined Betti numbers, see \cite{Buchberger1985} for an early algorithm for computing Groebner basis, and \cite[p.281-321]{Eisenbud2001} for examples of how they are used to compute refined Betti numbers.
In general, these methods do not apply to real algebraic varieties.
Also, these calculations may take a prohibitively long time in some cases, see e.g. \cite[Section 4.4]{Aggarwal2024}.

In this section, we present an alternative way to use the reach to compute the Betti numbers of a curve $M=Z(f)$ in a bounded set of $[-L,L]^2 \subset \R^2$.
We start by covering $[-L,L]^2$ with closed boxes whose side length is small compared to the reach, and make the following definition:

\begin{definition}
\label{definition:selected-box}
    We say that a box $b$ is \emph{selected}, if there exist two vertices $v_1,v_2$ of $b$, such that $f(v_1) \geq 0$ and  $f(v_2) \leq 0$.
\end{definition}

By the mean value theorem, every selected box $b$ must contain a point of $M$, i.e. $b \cap M \neq \emptyset$.
However, the converse is not true:
a box may be not selected, but still contain a point of $M$.
This situation is shown in \cref{figure:unselected-box}.
The surprising result proved in this section is that the homology groups of $M$ and the union of all selected boxes agree nonetheless.
Because checking whether a box is selected or not is very simple, this allows for fast computation of homology groups of $M$.

\begin{figure}[htbp]
\centering
\includegraphics[width=12cm]{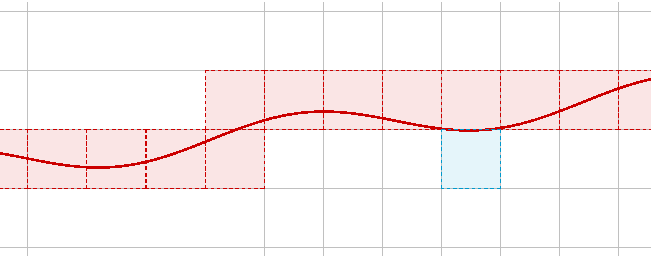}
\caption{Picture of a variety $M=Z(f) \subset \R^2$ in red. Red shaded boxes are \emph{selected} according to \cref{definition:selected-box}. The blue shaded box contains a point of $M$ but has not been selected. The conclusion of \cref{corollary:boxes-and-M-have-same-homology-groups} remains true: the union of all selected boxes has the same Betti numbers as $M$.}
\label{figure:unselected-box}
\end{figure}

Now let $M \subset [-L,L]^2 \subset \R^2$ be decomposed into closed boxes of side length 
\begin{align}
\label{equation:delta-choice}
\delta \leq \tau/2.37,
\end{align}
so that two adjacent closed boxes intersect in a common face.
Let
\begin{align}\label{def_B}
    B & :=
    \bigcup_{b \text{ selected}}
    b,
    \\
    \tilde{B} & :=
    \bigcup_{b \cap M \neq \emptyset}
    b.
\end{align}

\begin{proposition}
    \label{proposition:M-B-tilde-deformation-retract}
    The set $M$ is a deformation retract of $\tilde{B}$.
\end{proposition}

\begin{proof}
    Given a box $b \in \tilde{B}$, since the box diagonal is less than the reach $\tau$, we have that for all $p \in b$ there exists a unique nearest point $\pi(p)$ on $M$ for $p$.
    We first claim that $\pi(p)$ must be in a box adjacent to $b$, i.e. $\pi(p) \in b'$ with $b \cap b' \neq \emptyset$, so although the two boxes need not have an edge in common, they must have at least one vertex in common.

    Assume that $b'$ is not adjacent to $b$.
    Let $x:= \pi(p)$ and let $q \in b \cap M$ be the nearest point to $p$ in $b$.
    Then $d(p,x) < d(p,q) \leq \sqrt{2}\delta$, because $x$ is closer to $p$ than $q$, and $p$ and $q$ are in the same box, hence $d(q,x) \leq d(q,p)+d(p,x) \leq 2 \cdot \sqrt{2} \delta$.
    On the other hand:
    \begin{align*}
        \angle(q-x, p-x) \leq 2 \arctan (1/2),
    \end{align*}
    because the angle $\angle(q-y, p-y)$ for any points $p,q \in b$ and $y$ in any box not adjacent to $b$ is maximal in the situation of \cref{fig:box-example}, where it is $2 \cdot \arctan(1/2) \approx 0.93$.

    \begin{figure}[htbp]
        \centering
        \includegraphics[width=7cm]{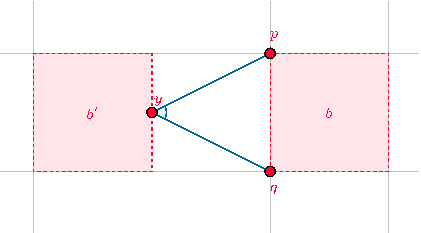}
        \caption{A constellation of points attaining the maximum value for $\angle(q-y, p-y)$, namely $\angle(q-y, p-y) = 2 \arctan (1/2)$.}
        \label{fig:box-example}
    \end{figure}

    Hence,
    \begin{align*}
        \angle (q-x, T_x M)
        \geq
        \angle(p-x, T_x M)
        -
        \angle(q-x, p-x)
        \geq
        \pi/2-2 \arctan(1/2)
        \approx
        0.64,
    \end{align*}
    where we used that $p-x$ is orthogonal to $T_x M$.
    Thus, by \cite[Theorem 2.2]{Aamari2019}:
    \begin{align*}
    \tau
    &\leq
    \frac{|q-x|}{2d((q-x)/|q-x|, T_x M)}
    \\
    &\leq
    \frac
    {2 \cdot \sqrt{2} \delta}
    {2 \sin (\angle(q-x, T_x M))}
    \\
    &\leq
    \frac
    {2 \cdot \sqrt{2} \delta}
    {2 \sin(\pi/2-2\arctan(1/2))}
    \\
    &\approx 2.36 \cdot \delta,
    \end{align*}
    which is a contradiction due to our choice of $\delta$ from \cref{equation:delta-choice}.
    Thus, the assumption was wrong, and $b'$ \emph{is} adjacent to $b$.
    Denote by $m := \{x \in \R^2: x \text{ is the midpoint of a box}\}$.
    Define the map 
    \begin{align*}
        \text{push}
        :
        \R^2 \setminus ( m \setminus \tilde{B})
        &\rightarrow \R^2
        \\
        x & \mapsto
        \begin{cases}
            x, & \text{ if } x \in \tilde{B}\\
            \overrightarrow{z x} \cap \partial b',
            & \text{ if }
            x \text{ lies in a box } b' \text{ with centre } z,
        \end{cases}
    \end{align*}
    i.e. $\text{push}$ is the identity on $\tilde{B}$ and pushes the content of every other box to its boundary, which is only well-defined away from its midpoint.
    Define the homotopy
    \begin{align*}
        H:[0,1] \times \tilde{B} & \rightarrow \R^N
        \\
        (t,p) & \mapsto \gamma_p(t),
    \end{align*}
    where $\gamma_p(t)$ is the straight line from $p$ to $\pi(p)$ with $\gamma_p(0)=p$ and $\gamma_p(1)=\pi(p)$.
    The map $H$ need not define a deformation retract from $\tilde{B}$ to $M$ because its image need not be contained in $\tilde{B}$.

    We showed before that $p$ and $\pi(p)$ lie in adjacent boxes, so that $\Im(H)$ is contained in the domain of $\text{push}$.
    Thus, $\text{push} \circ H$ is well defined, and is a composition of two continuous maps, so it defines a deformation retract from $\tilde{B}$ to $M$, which proves the claim.
\end{proof}

\begin{proposition}
    \label{proposition:not-selected-box-has-at-most-one-nonempty-face}
    If a box $b\in \tilde{B}$ is not selected, then there exists at most one face $F$ of $b$ such that $F \cap M \neq \emptyset$.
\end{proposition}

\begin{proof}
    Assume there are two faces, $F,F'$, of $b$ which have a non-empty intersection with $M$.
    There are two possible cases, shown in Fig. \ref{fig:2cases}.

    \begin{figure}[htbp] \label{fig:2cases}
        \centering
        \begin{minipage}{0.45\textwidth}
        \centering
        \includegraphics[width=5cm]{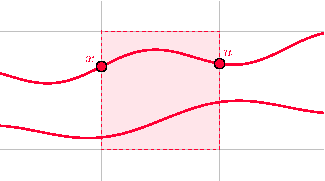}
        \end{minipage}
        \begin{minipage}{0.45\textwidth}
        \centering
        \includegraphics[width=5cm]{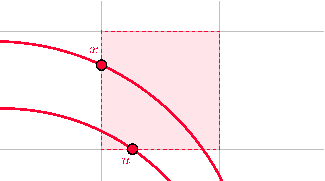}
        \end{minipage}
        
        \caption{Assume a box which contains at least a point of $M$ and it is not selected, has two faces with a non-empty intersection with $M$. There are two possible cases, which are treated separately in the proof of \cref{proposition:not-selected-box-has-at-most-one-nonempty-face}. Case 1, in which $M$ meets opposite faces of the box, is displayed on the left. Case 2, in which $M$ meets adjacent sides of the box, is displayed on the right.}
        \label{fig:case1-case2-face-intersections}
    \end{figure}

    \textbf{Case 1:
    $F$ and $F'$ are opposite faces of $b$.}

    Because $b$ is not selected, we must have that $M$ meets each face $F$ in an even number of points, or there exists a point $x \in F$ where $M$ touches $F$, i.e. $T_x M=F$.
    Let $x \in F \cap M$, and let $u \in F' \cap M$.
    We have
    \[
        d((u-x)/|u-x|, T_x M)
        =
        \sin (\angle (u-x, T_x M)),
    \]
    therefore, by \cite[Theorem 2.2]{Aamari2019} again,
    \begin{align*}
        \tau
        &\leq
        \frac{|u-x|}
        {2d((u-x)/|u-x|, T_x M)}
        \\
        &\leq
        \frac{\sqrt{2} \cdot \delta}{2\sin (\angle (u-x, T_x M))},
    \end{align*}
    so
    \begin{align}
        \label{equation:angle-smallness-condition}
        \begin{split}
        \angle (u-x, T_x M)
        &\leq
        \arcsin ((\sqrt{2}\delta)/(2\tau))
        \\
        \text{and analogously }
        \quad
        \angle (F, T_x M)
        &\leq
        \arcsin (\delta/(2\tau)).
        \end{split}
    \end{align}
    On the other hand, 
    \begin{align*}
        \angle (F,u-x) &> \pi/4,
    \end{align*}
    hence by the triangle inequality
    \begin{align*}
        \pi/4
        &<
        \angle (F,u-x)
        \\
        &\leq
        \angle (u-x, T_x M)+\angle (F, T_x M)
        \\
        &\leq
        \arcsin (\delta/2\tau)+\arcsin ((\sqrt{2}\delta)/(2\tau)),
    \end{align*}
    which implies $\frac{\delta}{\tau} \geq 0.63$, which again is a contradiction to our choice of $\delta$ from \cref{equation:delta-choice}.

    \textbf{Case 2:
    $F$ and $F'$ are neighbouring faces of $b$.}

    As before, let $x \in F$ and $u \in F'$.
    At least one of the following must be true:
    \begin{align*}
        \angle (u-x, F) \geq \pi/4
        \quad
        \text{ or }
        \quad
        \angle (x-u, F') \geq \pi/4.
    \end{align*}
    Without loss of generality assume the first case.
    As in \cref{equation:angle-smallness-condition}, we have that
    \begin{align*}
        \angle (u-x, T_x M) & \leq \arcsin ((\sqrt{2}\delta)/(2\tau)),
        \\
        \angle (F, T_x M) & \leq \arcsin (\delta/(2\tau)),
    \end{align*}
    hence
    \begin{align*}
        \pi/4
        &\leq
        \angle (u-x, F)
        \\
        &\leq
        \angle (u-x, T_x M) + \angle (F, T_x M)
        \\
        &\leq
        \arcsin ( \delta/(2 \tau))+\arcsin ((\sqrt{2}\delta)/(2\tau)),
    \end{align*}
    which is a contradiction for our choice of $\delta$ from \cref{equation:delta-choice} like in case 1.
\end{proof}

\begin{corollary}
    \label{corollary:box-unique-face}
    If a box $b \in \tilde{B}$ is not selected, then there exists a unique face $F$ such that $F \cap M \neq \emptyset$.
    Let $b'$ be the neighbouring box of $b$ with the property $b \cap b' = F$.
    Then $b'$ has been selected.
\end{corollary}

\begin{proof}
    By \cref{proposition:not-selected-box-has-at-most-one-nonempty-face} there exists at most one face $F$ of $b$ such that $F \cap M \neq \emptyset$.
    There must also exist at least one, since otherwise there would be a connected component of $M$ contained in $b$.
    In this case, it would have a bottleneck of length at most $\sqrt{2} \delta$, which contradicts \cref{theorem:reach-characterisation} for our choice of $\delta$ from \cref{equation:delta-choice}.

    Thus, let $b'$ be the uniquely determined neighbouring box such that $b \cap b' \cap M \neq \emptyset$.
    If $b'$ is not selected, then by \cref{proposition:not-selected-box-has-at-most-one-nonempty-face}, the only face of $b'$ with non-empty intersection with $M$ is the one connected to $b$.
    Thus, $b \cup b'$ contains a connected component $C$ of $M$.
    Fix $x \in C \cap b \cap b'$, and let $y=\argmax _{z \in C}d(x,z)$.
    Then $x-y \perp T_y M$ and $d(x,y) \leq \sqrt{2} \delta$, so by \cite[Theorem 2.2]{Aamari2019} we have that $\tau \leq \frac{1}{2} \sqrt{2} \delta$, which is a contradiction to our choice of $\delta$ from \cref{equation:delta-choice}.
\end{proof}

\begin{corollary}
    \label{corollary:B-Btilde-deformation-retracts}
    The set $B$ is a deformation retract of $\tilde{B}$.
\end{corollary}

\begin{proof}
    First note that $B \subset \tilde{B}$, because if $f(v_1) \geq 0$ and $f(v_2) \leq 0$ for vertices $v_1,v_2$ in a box, then there must exist a point $p$ in this box where $f(p)=0$ by the mean value theorem.

    We now define a homotopy from $\tilde{B}$ onto $B$.
    To this end, define 
    \begin{align*}
        q: \tilde{B} & \rightarrow B
        \\
        b \ni x & \mapsto \argmin _{y \in F} d(x,y) \text{ with $F$ face of $b$ s.t. $F \cap M \neq \emptyset$},
    \end{align*}
    i.e. $q$ projects an unselected box onto the uniquely determined face (\cref{corollary:box-unique-face}) that has a non-empty intersection with $M$.
    Then let
    \begin{align*}
        H:[0,1] \times \tilde{B} &\rightarrow \tilde{B}
        \\
        (t, x) & \mapsto
        \begin{cases}
            x, &\text{ if } x \in B
            \\
            t \cdot q(x)+(1-t) \cdot x, &\text{ if } x \in \tilde{B} \setminus B.
        \end{cases}
    \end{align*}
    Then $H$ is continuous because $q$ is continuous, and $\Im H=B$, which proves the claim.
\end{proof}

\begin{corollary}
\label{corollary:boxes-and-M-have-same-homology-groups}
    We have that $\pi_1(M) \cong \pi_1(B) \cong \pi_1(\tilde{B})$ and $H^i(M;\mathbb{Z}) \cong H^i(B;\mathbb{Z}) \cong H^i(\tilde{B};\mathbb{Z})$ for $i\in \{0,1\}$.
\end{corollary}

\begin{proof}
    Deformation retracts have the same fundamental group and cohomology groups, so this claim follows from \cref{proposition:M-B-tilde-deformation-retract,corollary:B-Btilde-deformation-retracts}.
\end{proof}

In order to make use of these results in applications, it remains to compute the homology group of the complex of selected boxes $B$.
An efficient algorithm for this problem was explained in \cite{Pilarczyk2015}, but we leave an implementation of our algorithm combined with a homology computation algorithm for future work.

\section{Application: lowest eigenvalue estimates for the Laplacian on algebraic varieties}
\label{section:application-lowest-eigenvalue-bound}

It is a common problem in analysis to prove injectivity estimates for linear differential operators, which is essentially equivalent to finding a lower bound for its first non-zero eigenvalue.
For example, explicit lower bounds for linear PDEs can be used in numerically verified proofs for the existence of solutions to PDEs, see e.g. \cite[Section 4]{Nakao2019}.

One of the most prominent linear differential operators is the Laplace operator, and there are many interesting problems in geometry whose linearisation is this operator.
If one were to carry out a numerically verified proof for such a problem, it would be one crucial step to obtain a lower bound for its first non-zero eigenvalue.

There is a vast amount of literature concerning lower bounds for the first non-zero eigenvalue of the Laplacian.
Many results have certain curvature requirements, and a review can be found in \cite{Ling2010}.
A milestone result that will be applied later in this section that does not make any assumptions about positivity or negativity of curvature is \cite{Li1980}.

There are applications of such estimates for submanifolds of $\R^N$:
specific work for submanifolds of the sphere (which are in particular submanifolds of $\R^N$) is \cite{Choi1983}, related to a conjecture of Yau.
An example for a submanifold of $\CP^N$ (which is $U(1)$-covered by a submanifold of $S^{2N+1} \subset \R^{2N}$) is \cite[Theorem 1.2 and Section 7, Example 2]{Li2012}.

One motivation for this article is to obtain lower bounds of the first eigenvalue of the Laplacian for Calabi-Yau varieties, in order to get some rigorous estimates for the difference between the true unknown Calabi-Yau metric and certain approximate Calabi-Yau metrics, see \cite[Section 3.1]{Douglas2024}.

The goal of this section is to give a lower bound of the first eigenvalue of the Laplacian on a submanifold $M \subset \R^N$, depending only on its reach, the ambient dimension $N$, and its diameter measured in the metric $| \cdot |_{\R^N}$ of the ambient space.
The reach can be estimated using the result from \cref{section:lower-bound-for-reach}.
The diameter in $\R^N$ is easily estimated from above, and we have already made use of this when we assumed that $M \subset [-M_1,M_1]^N$.

The eigenvalue estimate is obtained by bounding the diameter (measured in $d_M$) from above, bounding the Ricci tensor from below, and then applying \cite[Theorem 7]{Li1980} which immediately gives the eigenvalue bound.

We begin by recalling the following theorem giving a bound for the number of closed balls needed to cover a larger ball:

\begin{theorem}[p.163-164 in \cite{Rogers1963} and Theorem 1.1 \cite{VergerGaugry2005}]
\label{theorem:number-of-balls-needed-to-cover-big-ball}
    Let $T > \frac{1}{2}$ and $\nu_{T,N}$ be the minimal number of closed balls of radius $\frac{1}{2}$ which can cover a closed ball of radius $T$ in $\R^N$, $N \geq 2$.
    \begin{enumerate}
        \item 
        If $n \geq 3$, with $\vartheta_N:=N \log N+N \log ( \log N)+5N$, we have
        \begin{align}
            1
            <
            \nu_{T,N}
            \leq
            \begin{cases}
                e \vartheta_N(2T)^N & \text{ if } T \geq \frac{N}{2}
                \\
                N \vartheta_N(2T)^N & \text{ if } \frac{N}{2 \log N} \leq T < \frac{N}{2}.
            \end{cases}
        \end{align}

        \item 
        If $n \geq 9$ we have
        \begin{align}
            1
            <
            \nu_{T,N}
            \leq
            \frac{4e(2T)^N N \sqrt{N}}{\log N-2}
            \left(N \log N+N \log(\log N)+N \log (2T)+\frac{1}{2} \log(144 N) \right)
        \end{align}
        for all $1/2<T<\frac{N}{2 \log N}$.
    \end{enumerate}
\end{theorem}

This immediately gives a very naive bound for the diameter of $M$ as follows:
suppose we cover $M$ with small balls $B_1,\dots,B_n$ (small compared to the reach of $M$). In that case, we have an estimate for the diameter of $M \cap B_i$ for $i \in \{1,\dots,n\}$ with respect to $d_M$, using our comparison formula for extrinsic and intrinsic distances from \cref{application:extrinsic-instric-distance-estimate}.

Now, the longest a closed curve on $M$ can be is if it traverses each of the balls covering $M$, and because we can bound its length in each ball, this implies an overall diameter bound.
This is made rigorous in the following corollary:

\begin{corollary}
    \label{corollary:diameter-bound}
    Let $M=Z(f) \subset \mathbb{R}^N$ be connected and $\tau$ be the reach of $M$.
    Assume $M \subset B(K)$ for a ball with radius $K$ centered around some point in $\R^N$.
    Then
    \begin{align}
        \diam(M)
        :=
        \max_{x,y \in M} d_M(x,y)
        \leq
        2 (\nu_{K/\tau,N}+1)
        \cdot
        \tau.
    \end{align}
\end{corollary}

\begin{proof}
    By \cref{theorem:number-of-balls-needed-to-cover-big-ball}, we have that $B(K)$ can be covered by $n:=\nu_{K/\tau,N}$ balls of radius $\tau/2$ that we write as $B_1,B_2,\dots,B_n$.
    We also denote the midpoint of $B_i$ as $\mu_i$ for $i \in \{1,\dots,n\}$.
    
    Let $x,y \in M$ such that $d_M(x,y)=\diam(M)$ and let $m \in \mathbb{Z}$ and $k_1,\dots,k_m$ be a sequence of indices with the following properties:
    \begin{itemize}
        \item 
        $x \in B_{k_1}$ and $y \in B_{k_m}$;

        \item 
        $B_{k_i} \cap B_{k_{i+1}} \neq \emptyset$ for $i \in \{1,\dots,m-1\}$;

        \item 
        the number $m$ is the minimal number with the property that such a sequence of indices exists.
    \end{itemize}
    Such a sequence exists because $M$ is connected.
    This then satisfies $m \leq n$.
    Choose some $y_i \in B_{k_i} \cap B_{k_{i+1}}$ for $i \in \{1,\dots,m-1\}$.
    For two points $p,q \in B_k$ for $k \in \{1,\dots,n\}$ we have
    \begin{align}
        \label{equation:distance-in-ball}
        d_M(p,q)
        &\leq
        d_M(p,\mu_k)+d_M(\mu_k,q)
        \leq
        2 |p-\mu_k|_{\R^N}+2 |\mu_k-q|_{\R^N}
        \leq
        4 \frac{\tau}{2}
        =
        2 \tau,
    \end{align}
    where we used \cref{corollary:intrinsic-extrinsic-distance-estimate} in the second step.
    Hence
    \begin{align*}
        d_M(x,y)
        &\leq
        d_M(x,y_1)+d_M(y_m,y)+
        \sum_{i=1}^{m-1}
        d_M(y_i,y_{i+1})
        \\
        &\leq
        (m+1) \cdot 2 \cdot \tau
        \\
        &\leq
        2(n+1)\tau,
    \end{align*}
    where we used \cref{equation:distance-in-ball} in the second step.
\end{proof}

We now give an upper bound for the intrinsic curvature of $M$.
It is a classical result in Riemannian geometry that this intrinsic curvature is determined by the extrinsic curvature (as measured by the second fundamental form $\operatorname{II}$ from \cref{definition:second-fundamental-form}) and the curvature of the ambient manifold, in our case $\R^N$ which has curvature zero:

\begin{theorem}[Gauss equation, p.100 in \cite{ONeill1983}]
    Let $M$ be a semi-Riemannian submanifold of $\overline{M}$, with $R$ and $\bar{R}$ their respective Riemannian curvature tensors and $\operatorname{II}$ the second fundamental form.
    Then for vector fields $V,W,X,Y$ all tangent to $M$,
    \begin{align}
        \label{equation:gauss-equation}
        \begin{split}
        \< R(V,W)X,Y \>
        &=
        \< \bar{R}(V,W)X,Y\>
        +
        \<\operatorname{II}(V,X),\operatorname{II}(W,Y)\>
        \\
        &\quad
        - \< \operatorname{II}(V,Y),\operatorname{II}(W,X) \>.
        \end{split}
    \end{align}
\end{theorem}

\begin{corollary}
    Let $M \subset \R^N$ be an $n$-dimensional submanifold with reach $\tau$ and Riemannian curvature tensor $R$.
    Then for any unit vectors $e_1,e_2,e_3,e_4$ (not necessarily distinct or orthogonal) in $T_pM$,
    \begin{equation}
        |\<R(e_1,e_2)e_3,e_4\>|
        \leq
        \frac{9}{2} \tau^{-2}
    \end{equation}
    and in particular, 
    \begin{equation}
        |\Ric(e_1,e_2)|
        \leq
        \frac{9}{2}(n-1)\tau^{-2}.
    \end{equation}
\end{corollary}

\begin{proof}
    For simpler notation we write $|\cdot|_{\R^N}=|\cdot |$.
    Using the polarization formula and \cref{theorem:reach-characterisation}, we have for unit length vectors $e_1, e_2$:
    \begin{align*}
        |\operatorname{II}(e_1, e_2)|
        &\leq
        \frac{1}{4}
        \left(
        |\operatorname{II}(e_1,e_1)|+
        |\operatorname{II}(e_2,e_2)|+
        |\operatorname{II}(e_1-e_2,e_1-e_2)|
        \right)
        \\
        &\leq
        \frac{1}{4}
        \left(
        \tau^{-1}
        +
        \tau^{-1}
        +
        |e_1-e_2|^2 \tau^{-1}
        \right)
        \\
        &\leq
        \frac{1}{4} \tau^{-1} (1+1+2^2)
        \\
        &\leq \frac{3}{2} \tau^{-1}.
    \end{align*}
    Write $\bar{R}$ for the Riemannian curvature tensor on $\R^N$.
    Because $\bar{R} \equiv 0$ we have by \cref{equation:gauss-equation} for unit vectors $e_1,e_2,e_3,e_4$:
    \begin{align*}
        |\<R(e_1,e_2)e_3,e_4\>|
        \leq
        2 \cdot \frac{9}{4} \cdot \tau^{-2}
        =
        \frac{9}{2} \tau^{-2}.
    \end{align*}
    Therefore
    \begin{align*}
        |\Ric(e_1,e_2)|
        \leq
        \sum_{i=1}^n
        |\<R(e_1,E_i)E_i,e_2\>|
        \leq
        \frac{9}{2}(n-1)\tau^{-2},
    \end{align*}
    where $E_1=e_1,E_2,E_3,\dots,E_n$ is an orthonormal basis of $T_p M$.
    In the last step, we obtain a factor of $(n-1)$ rather than $n$ because at least one of the summands is zero based on the symmetry properties of $R$.
    (If $e_1 \neq e_2$ for two orthonormal vectors $e_1,e_2$, then at least two summands are zero.)
\end{proof}

We now obtain a lower bound for the first eigenvalue of the Laplacian as follows:

\begin{theorem}[Theorem 7 in \cite{Li1980}]
    Let $M$ be a compact $n$-dimensional manifold with Ricci curvature of $M$ bounded below by $(n-1)\xi$ and diameter $d$.
    If $\lambda_1$ is the first non-zero eigenvalue of the Laplace operator acting on functions, then
    \begin{align}
        \label{equation:eigenvalue-estimate}
        \lambda_1
        \geq
        \frac{
            \exp \left( -(1+(1-4(n-1)^2d^2\xi)^{1/2}) \right)
        }
        {
            2(n-1)d^2
        }.
    \end{align}
\end{theorem}

\begin{corollary}
    \label{corollary:first-eigenvalue}
    Let $M \subset \R^N$ be an $n$-dimensional submanifold with reach $\tau$ contained in a ball of radius $K$.
    Let
    \begin{align*}
        \xi &
        :=
        -\frac{9}{2}\tau^{-2}
        \text{ and }
        d
        :=
        2(\nu_{K/\tau,N}+1) \cdot \tau,
    \end{align*}
    where $\nu_{K/\tau,N}$ was defined in \cref{theorem:number-of-balls-needed-to-cover-big-ball}.

    If $\lambda_1$ is the first non-zero eigenvalue of the Laplace operator acting on functions, then a lower bound for $\lambda_1$ is given by the expression \cref{equation:eigenvalue-estimate} for these values of $\xi$ and $d$.    
\end{corollary}

Unfortunately, the results of this section are of limited use in applications, as the following example shows:

\begin{example}
    We apply this to the example $f(x,y)=x^2+y^2-1$ and $M:= Z(f) \subset \R^2$, i.e. the circle in $\R^2$.
    It is known that the smallest non-zero eigenvalue on it is $1$, but the bound from \cref{corollary:first-eigenvalue} turns out to be much smaller than this.
    
    In \cref{example:circle} we obtained $|\nabla f|_1 \geq 0.3535$ on $M$ and $\tau \geq 0.0625$ in \cref{example:circle-reach}.
    In turn, using that $M \subset [-K,K]^3$ for $K=2$ this gives:
    \begin{align*}
        \nu_{K/\tau,N}
        =
        118614
    \end{align*}
    and using \cref{corollary:first-eigenvalue} we obtain
    \[
    \lambda_1
    \geq
    \frac{\exp ( -1.01 \cdot 10^6)}{4.40 \cdot 10^8},
    \]
    which is far smaller than the smallest positive float number.
\end{example}

The reason this estimate is so far from the true value is that our diameter estimate from \cref{corollary:diameter-bound} is so far from the true value.
It remains an interesting problem to compute rigorous upper bounds for the diameter of a given manifold that are sharper than the one we provide.

\section{Application: deforming algebraic varieties without changing their diffeomorphism type}
\label{section:application-deforming-algebraic-varieties}

In this section, we consider algebraic varieties of the form $Z(f) \subset \R^N$, where $f$ is a \emph{polynomial}.
Previously, we never used any particular property of polynomials, only that $f$ is smooth.
In what follows we will restrict to the case $f \in \R[x_1,\dots,x_N]$.

The condition for $M$ to be smooth is that $\nabla f \neq 0$ everywhere on $M$.
This is an open condition, so for any other polynomial $P \in \R[x_1,\dots,x_N]$, there exists $\epsilon_0 >0$ such for all $0 < \epsilon < \epsilon_0$ we have that $X_\epsilon := Z(f+\epsilon P)$ is still smooth.
In particular, $X$ and $X_\epsilon$ are diffeomorphic as a consequence of the Ehresmann lemma \cite[Lemma 9.2]{Kolar1993}.
In applications, it may be necessary to find a lower bound for $\epsilon_0$, see for example \cite[Section 3.2]{Douglas2024} where it was needed for a projective variety.

The following proposition provides a lower bound for $\epsilon_0$.
It is proved in an analogous way to the case of projective varieties in \cite[Proposition 3.12]{Douglas2024}.

\begin{proposition}
    Let $B \subset U \subset \R^N$ and $f,P \in \R[x_1,\dots,x_N]$.
    Define
    \[
    \delta
    :=
    \min_{x \in B}
    | \grad f (x) |_{\R^N}
    \quad
    \text{ and }
    \quad
    \xi
    :=
    \min_{x \in U \setminus B}
    |f(x)|
    \]
    and assume $\xi >0$.
    If $\|{P}_{C^0(U \setminus B)} \leq \xi$ and $\|{\grad P}_{C^0(B)} \leq \delta$, then
    $M_\epsilon := Z(f+\epsilon P) \cap U$ is smooth for all $\epsilon \in [0,1)$.
\end{proposition}

\begin{proof}
    Let $x \in U \cap M_{\epsilon}$, we divide the proof in two steps,

    \textbf{Step 1: show that $x \in B$.}
    Assume $x \in U \setminus B$.
    Then
    \[
    |(f+\epsilon P)(x)|
    \geq
    |f(x)|-\epsilon |P(x)|
    \geq
    \xi-\epsilon \xi
    >0,
    \]
    i.e. $x \notin M_{\epsilon}$.

    \textbf{Step 2: show that $M_{\epsilon}$ is smooth at $x$.}
    By Step 1 we have that $x \in B$, therefore:
    \[
    |\grad(f+\epsilon P)(x)|_{\R^N}
    \geq
    |\grad f(x)|_{\R^N}
    -\epsilon |\grad P(x)|_{\R^N}
    \geq
    \delta-\epsilon \delta
    >0,
    \]
    i.e. $x$ is a smooth point of $M_\epsilon$.
\end{proof}

\begin{example}
    We revisit \cref{example:circle}.
    \Cref{algorithm:lower-bound-for-df} provides the set $B$, which is the union of CaseTwoBoxes.
    In the notation from \cref{algorithm:lower-bound-for-df}, we have $M_2=5.66$, $M_3=2$ and from the algorithm we obtain the smallest box length to be $\epsilon=0.0625$, so by \cref{proposition:df-algorithm-correctness}:
    \[
    \min_{x \in B}
    |\nabla f(x)|_2
    \geq
    \frac{1}{\sqrt{2}}
    \min_{x \in B}
    |\nabla f(x)|_1
    \geq
    \frac{1}{\sqrt{2}}
    M_3 N^{3/2} \frac{0.0625}{2}
    >
    0.12
    \]
    and
    \[
    \min_{x \in [-2,2]^2 \setminus B}
    |f(x)|
    \geq
    \sqrt{N} \frac{\epsilon}{2} M_2
    >
    0.25.
    \]
    Hence, for example for the polynomial $P \equiv 0.12$ we have that $M_\epsilon:=Z(f+\epsilon P)=Z(x^2+y^2-1+\epsilon)$ is smooth for $\epsilon \in [0,1)$.
    In fact, $M_\epsilon$ is even smooth for $\epsilon \in [0, 1/0.12)$, but our obtained lower bound is consistent with this.
\end{example}

\pagebreak

\appendix

\section{Manifolds $M=Z(f_1,\dots,f_k)$ defined by multiple functions}
\label{section:generalise-to-f1-...-fk}

In \cref{section:lower-bound-for-df} it was explained how to compute a lower bound for $|\nabla f|_1$ on a variety $M=Z(f) \subset \R^N$.
In this section, we present an analogue result for varieties defined by more than one function, say $M=Z(f_1,\dots,f_k)$.

In this case, a point $x \in M$ is singular, if the linear map $(\nabla f_1, \dots, \nabla f_k)$ does not have full rank.
For $k=1$, this reduces to the condition $\nabla f_1(x)=0$.

This makes the situation much more complicated:
our goal is still to compute some data for $|\nabla f_i|_1$ and deduce a lower bound for the reach from it.
However, lower bounds for $|\nabla f_i|_1$ will not be enough:
it may happen that all $|\nabla f_i|_1$ are large, but $|\nabla f_i|_1=|\nabla f_j|_1$ for some $i \neq j$.
In this case, $M$ would be singular, which heuristically means that its reach is $\tau=0$.
More rigorously, one may find that the $|\nabla f_i|_1$ are large, but $|\nabla f_i|_1$ and $|\nabla f_j|_1$ are very close, which would make the reach $\tau$ very small.

In this more complicated situation, the interesting quantity is $\det g$, where the entries of the matrix $g \in
C^\infty(\R^N, \R^{k \times k})$ are
\begin{align}
\label{equation:def-of-matrix-g}
(g_{ij}) = \< \nabla f_i, \nabla f_j \> .
\end{align}

In \cref{subsection:computing-a-lower-bound-for-det-g}, we will explain an algorithm to compute a lower bound for $\det g$.
This is \cref{algorithm:lower-bound-for-det}, which is the analogue of \cref{algorithm:lower-bound-for-df} from \cref{section:lower-bound-for-df}.
Next, we will use this lower bound to compute a lower bound for the reach of $M$.
Mirroring \cref{section:lower-bound-for-reach}, we will do this in three parts:
first, we calculate an upper bound for the curvature in \cref{subsection:use-det-g-for-upper-bound-of-curvature}.
Second, computing a lower bound for the smallest bottleneck in \cref{subsection:det-g-for-smallest-bottleneck}.
Third, taking both together, we get a lower bound for the reach as an immediate corollary, which is stated in \cref{subsection:det-g-for-reach}.

\subsection{A lower bound for $\det g$ on $Z(f_1,\dots,f_k)$}
\label{subsection:computing-a-lower-bound-for-det-g}

We now generalise \cref{algorithm:lower-bound-for-df} as follows:

\begin{algorithm}
    \caption{Finding a lower bound for $\det g$, where $g$ was defined in \cref{equation:def-of-matrix-g}, on $Z(f_1,\dots,f_k)$. The function subdivide is as in \cref{algorithm:lower-bound-for-df}.}
    \label{algorithm:lower-bound-for-det}
    \begin{algorithmic}[1]
    \Require{$M_1,M_2,M_3$ such that $Z(f_1,\dots,f_k) \subset [-M_1,M_1]^N$ and for all $i \in \{1,\dots,k\}$ we have $|\nabla f_i|_{\R^N} \leq M_2$ and $|\Hess f_i|_2 \leq M_3$ on $[-M_1,M_1]^N$}
    \Statex
    \Let{NewBoxes}{$\{[-M_1,M_1]^N\}$}
    \Let{CaseOneBoxes}{$\emptyset$}
    \Let{CaseTwoBoxes}{$\emptyset$}

    \While{NewBoxes $\neq \emptyset$}
    \Let{CurrentBox}{pop(NewBoxes)}
    \Let{$m$}{midpoint(CurrentBox)}
    \Let{$\varepsilon$}{sidelength(CurrentBox)}
    \If{$\max_{j \in \{1,\dots,k\}} |f_j(m)| > \sqrt{N} \epsilon M_2$}
        \Let{CaseOneBoxes}{CaseOneBoxes $\cup \{\text{CurrentBox} \}$ }
    \ElsIf{$|\det g(m)|> 2N \epsilon \cdot k! \cdot M_2^{2k-1} M_3$}
        \Let{CaseTwoBoxes}{CaseTwoBoxes $\cup \{\text{CurrentBox} \}$ }
    \Else
        \Let{NewBoxes}{NewBoxes $\cup \text{subdivide}(\text{CurrentBox})$}
    \EndIf
  \EndWhile
  \end{algorithmic}
\end{algorithm}

\begin{proposition}
    If $Z(f_1,\dots,f_k)$ is non-singular, then \cref{algorithm:lower-bound-for-det} terminates after finitely many steps, and yields a set $B \subset \R^N$ with $Z(f_1,\dots,f_k) \subset B$ such that
    \begin{align}
        \left|
        \det g |_B
        \right|
        >
        N \epsilon \cdot k! \cdot M_2^{2k-1} M_3
    \end{align}
    as well as
    \begin{align}
        \max_{j \in \{1,\dots,k\}}
        \left|
        f_j|_{[-M_1,M_1]^N \setminus B}
        \right|
        >
        \sqrt{N} \frac{\epsilon}{2} M_2,
    \end{align}
    where $\epsilon$ is the side length of the smallest box in CaseOneBoxes$\,\cup\,$CaseTwoBoxes.
\end{proposition}

\begin{proof}
    Let $b \in \text{CaseOneBoxes}$ with side length $\epsilon$ and midpoint $m \in b$.
    Let $x \in b$ and $i \in \{1,\dots,k\}$ such that $|f_i(m)| = \max_{j \in \{1,\dots,k\}} |f_j(m)|$.
    We then see that $|f_i(x)|>0$ as in \cref{equation:algo-f-nonzero-argument}.

    Let $b \in \text{CaseTwoBoxes}$ with side length $\epsilon$ and midpoint $m \in b$.
    We aim to show that 
    \begin{align*}
        |\det g(x)|
        \geq
        N \epsilon k! M_2^{2k-1}M_3.
    \end{align*}
    As a shorthand, we write $u_i=\nabla f_i$.
    First note that for all $i,j\in \{1,\dots,k\}$ we have
    \begin{align}
        \label{equation:gij-estimates}
        \begin{split}
        |g_{ij}|
        &=
        \left| \< u_i, u_j \> \right|
        \leq
        |u_i|_{\R^N} \cdot |u_j|_{\R^N}
        \leq
        M_2^2,
        \\
        |\d g_{ij} |_{\R^N}
        &\leq
        \sqrt{N}
        |\d g_{ij} |_{\infty}
        =
        \sqrt{N}
        \left| \d g_{ij}(e_s) \right|
        =
        \sqrt{N}
        \left|
        \< \nabla_{e_s} u_i, u_j \rangle
        +
        \< u_i, \nabla_{e_s} u_j \rangle
        \right|
        \leq
        2 \sqrt{N} M_2 M_3
        \text{ for some }
        s \in \{1,\dots,N\},
        \end{split}
    \end{align}
    where $e_s=(0,\dots,0,1,0,\dots,0)$ denotes the $s$-th canonical basis vector of $\R^N$.
    This gives
    \begin{align*}
        \begin{split}
            |\d {} (\det g)|_{\R^N}
            &=
            \left| \d \left(
            \sum_{\sigma \in S_k} \sign(\sigma) g_{1 \sigma(1)} \dots  g_{k \sigma(k)}
            \right) \right|_{\R^N}
            \\
            &=
            \left| \left(
            \sum_{\sigma \in S_k} 
            \sign(\sigma) g_{1 \sigma(1)} \dots \widehat{g_{i\sigma(i)}} \dots g_{k \sigma(k)} \d {} (g_{i\sigma(i)})
            \right) \right|_{\R^N}
            \\
            &\leq
            k! \cdot 2 \cdot \sqrt{N} M_2^{2k-1} M_3,
        \end{split}
    \end{align*}
    where we used \cref{equation:gij-estimates} in the last step.
    Thus
    \begin{align*}
        | \det g(x)|
        &\geq
        | \det g(m)|-|m-x|_{\R^N} \cdot \max_{y \in [-M_1,M_1]^N} 
        \left| \d_y (\det g) \right|_{\R^N}
        \\
        &\geq
        | \det g(m)|-\sqrt{N} \frac{\epsilon}{2} \cdot k! \cdot 2 \cdot \sqrt{N} M_2^{2k-1} M_3
        \\
        &\geq
        N \epsilon \cdot k! \cdot M_2^{2k-1} M_3,
    \end{align*}
    where in the first step, we used the mean value theorem as in \cref{equation:algo-f-nonzero-argument}.

    Finally, the algorithm terminates after finitely many steps: Assume the contrary, then as with \cref{algorithm:lower-bound-for-df}, there exists an infinite sequence of closed boxes $\{\tilde{b}_i\}$ with monotone decreasing sides such that $p^* \in \cap^\infty_{i=1} \tilde{b}_i$. Then by line 8 in \cref{algorithm:lower-bound-for-det},
    \begin{equation*}
    |f_j(p^*)|=0 
    \end{equation*}
    for every $j\in \{1,\cdots,k\}$ and hence $p^*\in M$. On the other hand,
    \begin{equation*}
       |\det g (p^*)| = 0,
    \end{equation*}
And therefore $(\nabla f_1, \cdots, \nabla f_k)$ does not have full rank at $p^*$; hence, it is a singular point, contradicting the assumption that $M$ is smooth.
\end{proof}

\subsection{\texorpdfstring{Using $\det g$ for an upper bound of the curvature of $M$}{Using $\det g$ for an upper bound of the curvature of $M$}}
\label{subsection:use-det-g-for-upper-bound-of-curvature}

Armed with a lower bound for $\det g$, we will now mimic \cref{subsection-upper-bound-for-curvature-from-df} to compute an upper bound for the curvature of $M$.
We begin by stating the following easy linear algebra lemma:

\begin{lemma}
    \label{lemma:ONB-construction}
    Let $u_1, \dots, u_N \in \R^N$ be linearly independent.
    Define
    \begin{align*}
        (g_{ij})
        =
        \langle u_i, u_j \rangle
    \end{align*}
    and 
    \begin{align}
        \label{equation:v_i-ONB}
        v_i
        :=
        \sum_{k=1}^N
        \beta_{ik} u_k
        \quad
        \text{ with }
        \quad
        \beta_{ik}
        =
        \left(
        \sqrt{g^{-1}}
        \right)_{ik}.
    \end{align}
    Then $\langle v_i, v_j \rangle=\delta_{ij}$, i.e., $v_1,\dots,v_N$ is an orthonormal basis for $\R^N$.
\end{lemma}

\begin{proof}
    The matrix $g$ is the matrix representation of the Euclidean metric on the basis $\{u_i\}$, hence $g$ (and its inverse) is positive definite and has a unique positive definite square root.
    We have
    \[
        \langle v_i, v_j \rangle
        =
        \sum_{k,l=1}^N
        \beta_{ik} \beta_{jl} g_{kl}
        =
        \left(
        \sqrt{g^{-1}}
        g
        \sqrt{g^{-1}}
        \right)_{ij}
        =\delta_{ij}.\qedhere
    \]
\end{proof}

We now use this result to prove another lemma in linear algebra, namely bounding the norm of the inverse of a matrix by the reciprocal of its determinant, roughly speaking:

\begin{lemma}
    \label{lemma:sqrt-inverse-g-element-estimate}
    For any positive definite matrix $g \in \R^{N \times N}$ and for all $1 \leq i,k \leq N$ we have:
    \begin{align*}
        \left|\left(
        \sqrt{g^{-1}}
        \right)_{ik}\right|
        \leq
        N^{2+(N-1)/4} \frac{||g||_1^{(N-1)/2}}{|\det g|^{1/2}}.
    \end{align*}
\end{lemma}

\begin{proof}
    Denote by $||\cdot||_1$, $||\cdot||_2$, and $||\cdot||_F$ the $1$-norm, $2$-norm, and Frobenius norm for matrices respectively.
    Let $A$ be any symmetric, positive definite matrix and $\sqrt{A}$ be its uniquely defined positive definite square root.
    Denote by $\lambda_1,\dots,\lambda_N$ the eigenvalues of $A$, repeated according to their multiplicity.
    Then for any $i,j \in \{1,\dots, N\}$:
    \begin{align*}
        |(\sqrt{A})_{ij}|
        &\leq
        ||\sqrt{A}||_1
        \\
        &\leq
        \sqrt{N}
        ||\sqrt{A}||_{F}
        \\
        &=
        \sqrt{N}
        \sqrt{\sum_{i=1}^N \lambda_i}
        \\
        &\leq
        \sqrt{N}
        \sum_{i=1}^N \sqrt{\lambda_i}
        \\
        &\leq
        \sqrt{N} \cdot N \cdot \sqrt{||A||_{F}}
        \\
        &\leq
        N^{2} \sqrt{||A||_2}.
    \end{align*}
    In the first two steps, and the last one, we used the equivalence between the $||\cdot||_1$, $||\cdot||_F$ and $||\cdot||_2$ (see e.g. \cite[Theorem 1.6.2.6]{Geijn2023}), in the third step we used that the eigenvalues of $\sqrt{A}$ are $\sqrt{\lambda_1},\dots,\sqrt{\lambda_N}$, in the fourth step we used that the square root is subadditive and used $|\lambda_i| \leq ||A||_F$ for all $i \in \{1,\dots,N\}$ in the second-to-last step.
    Applying this to $A=g^{-1}$ we obtain
    \begin{align*}
        \left|\left(
        \sqrt{g^{-1}}
        \right)_{ik}\right|
        \leq
        N^{2} \sqrt{||g^{-1}||_2}
        \leq 
        N^{2} \sqrt{\frac{||g||_2^{N-1}}{\det g}}
        \leq
        N^{2+(N-1)/4} \frac{||g||_1^{(N-1)/2}}{(\det g)^{1/2}}.
    \end{align*}
    Here we used in the second step that for an invertible and diagonalisable matrix $A \in \R^{N \times N}$ with eigenvalues $\lambda_1 \leq \dots \leq \lambda_N$ we have
    \begin{align*}
        ||A^{-1}||_2
        =
        \frac{1}{|\lambda_1|}
        \leq
        \frac{1}{|\lambda_1|}
        \cdot
        \prod_{i=2}^N
        \frac{|\lambda_N|}{|\lambda_i|}
        =
        \frac{|\lambda_N|^{N-1}}{|\det A|}
        =
        \frac{||A||_2^{N-1}}{|\det A|}.
        \quad
        \qedhere
    \end{align*}
\end{proof}

We can now combine these linear algebra results to compute a bound for the second fundamental form of $M$ in terms of the reciprocal of $\det g$:

\begin{lemma}
    \label{lemma:det-g-for-second-FF-bound}
    Let $f_1,\dots,f_k: \R^N \rightarrow \R$ be smooth functions and let $u_i:=\grad f_i$ be the vector fields defined by their gradients, and define $M=Z(f_1,\dots,f_k)$. Assume that $u_1,\dots,u_k$ are linearly independent at each point of $M$.
    Denote by $g$ the smooth $k\times k$ matrix given by
    \[
    (g_{ij})
    :=
    \langle u_i, u_j \rangle 
    .
    \]
    Then for every $p$ in $M$ and every $v,w \in T_pM$, the second fundamental form of $M$ satisfies
    \begin{align}
        |\operatorname{II}(v,w)|_{\R^N}
        &\leq
        k^{3+(k-1)/4} \frac{||g_p||_1^{(k-1)/2}}{|\det g_p|^{1/2}}
        \sum_{i=1}^k
        |\Hess (f_i)(v,w)|.
    \end{align}
\end{lemma}

\begin{proof}
Let $X,Y$ be tangent vector fields on $M$.
Let $j \in \{1,\dots,k\}$.
Then it follows for the variety $M_j:=Z(f_j)$ defined by just one polynomial that
\begin{align}
    \label{equation:hessian-nabla-u}
    |\Hess (f_j)(X,Y)|
    &=
    \left|
    \frac{\Hess (f_j)(X,Y)}{|u_j|}
    \right|
    \cdot |u_j|
    =
    \left| \operatorname{II}_{M_j}(X,Y) \right| \cdot |u_j|
    =
    \left|
    \< \nabla_X Y, u_j \>
    \right|.
\end{align}

By \cref{lemma:ONB-construction}, the vectors $V_i= \sum_{j=1}^k\left(\sqrt{g^{-1}}\right)_{ij} u_j$ form an orthonormal frame of the normal bundle of $M$.
Using this basis, we get
\begin{align*}
    \begin{split}
    |
    \operatorname{II}(X,Y) 
    |_{\R^N}
    &=
    \left|
    \sum_{i=1}^k \langle \nabla_X Y, V_i\rangle V_i
    \right|_{\R^N}
    \\
    &=
    \left|
    \sum_{1\leq i,j \leq k} \left(
    \sqrt{g^{-1}}
    \right)_{ij} \langle \nabla_X Y, u_j \rangle V_i 
    \right|_{\R^N}
    \\
    &\leq
    \sum_{1\leq i,j \leq k} 
    \left|
    \left(
    \sqrt{g^{-1}}
    \right)_{ij} \Hess (f_j)(X,Y) V_i
    \right|_{\R^N},
    \end{split}
\end{align*}
where we used \cref{equation:hessian-nabla-u} in the last step.
The Lemma then follows immediately from \cref{lemma:sqrt-inverse-g-element-estimate}.
\end{proof}

\subsection{Using \texorpdfstring{$\det g$}{\det g}  for a lower bound for the smallest bottleneck}
\label{subsection:det-g-for-smallest-bottleneck}

We are still working towards generalising the results from sections \cref{section:lower-bound-for-df,section:lower-bound-for-reach}.
So far, we gave an algorithm to compute a lower bound for $\det g$ and used this lower bound to bound the second fundamental form of $M$.
The last thing remaining is to compute a lower bound for the smallest bottleneck of $M$, which will be executed in this section.
This generalises \cref{subsection:lower-bound-for-bottleneck}.

We begin with another linear algebra lemma:

\begin{proposition}
    \label{proposition:U-norm-equivalence}
    Let $u_1,\dots,u_k \in R^N$ be linearly independent satisfying $|u_i|_{\R^N} \leq M_2$ for all $i \in \{1,\dots,k\}$.
    Define the norm $|\cdot|_{U}$ on the $k$-dimensional vector space $\spann \{u_1,\dots,u_k\}$ as
    \[
        \left|
        \sum_{i=1}^k
        \alpha_i u_i
        \right|_U
        =
        \sum_{i=1}^k
        |\alpha_i|.
    \]
    We then have that for all $v \in \spann \{u_1,\dots,u_k\}$:
    \begin{align}
        \begin{split}
            \left| v \right|_{\R^N}
            &\leq
            M_2
            |v|_{U},
            \\
            \left| v \right|_{U}
            &\leq
            \Cl{const:U-to-R^N-norm-equivalence}
            |v|_{\R^N},
        \end{split}
    \end{align}
    where $\Cr{const:U-to-R^N-norm-equivalence}=N^{4+(N-1)/4} \frac{||g||_1^{(N-1)/2}}{|\det g|^{1/2}}$ for $g_{ij}=\langle u_i,u_j \rangle$.
\end{proposition}

\begin{proof}
    It is easy to check that $|\cdot|_U$ is a norm on $\R^N$. All norms on finite-dimensional vector spaces are equivalent, so the equivalence of $|\cdot|_{\R^N}$ and $|\cdot|_U$ is immediate, but it remains to find the explicit norm-equivalence constants.

    For $v=\sum_{i=1}^k \alpha_i u_i \in \R^N$ we have that
    \begin{align*}
        |v|_{\R^N}
        =
        \left| \sum_{i=1}^k \alpha_i u_i \right|_{\R^N}
        \leq
        \sum_{i=1}^k
        |\alpha_i| \cdot | u_i |_{\R^N}
        \leq
        M_2 |v|_{U}.
    \end{align*}
    Let $v_1,\dots,v_k$ be the orthonormal basis from \cref{equation:v_i-ONB} for $\spann \{u_1,\dots,u_k\}$.
    Then
    \begin{align*}
        |v|_{U}
        &=
        \left|
        \sum_{i=1}^k
        \langle v, v_i \rangle v_i
        \right|_U
        \\
        &=
        \left|
        \sum_{i,j=1}^k
        \left(\sqrt{g^{-1}}\right)_{ij}
        \langle v, v_i \rangle u_j
        \right|_U
        \\
        &=
        \sum_{j=1}^k
        \left|
        \sum_{i=1}^k
        \left(\sqrt{g^{-1}}\right)_{ij}
        \langle v, v_i \rangle
        \right|
        \\
        &\leq
        |v|_{\R^N} \underbrace{|v_i|_{\R^N}}_{=1}
        \sum_{ij}
        \left|
        \left(
        \sqrt{g^{-1}}
        \right)_{ij}
        \right|
        \\
        &\leq
        N^2 N^{2+(N-1)/4} \frac{||g||_1^{(N-1)/2}}{|\det g|^{1/2}}
        |v|_{\R^N},
    \end{align*}
    where we used Cauchy–Schwarz inequality in the first inequality and \cref{lemma:sqrt-inverse-g-element-estimate} in the last step.
\end{proof}

Using this, we can now deduce a lower bound for the smallest bottleneck of $M$.
Note the dependence on a lower bound of $\det g$ through the appearance of the constant $\Cr{const:U-to-R^N-norm-equivalence}$ in the following proposition:

\begin{proposition}
    \label{proposition:det-g-for-smallest-bottleneck}
    Let $M=Z(f_1,\dots,f_k) \subset U \subset \R^N$ for a convex set $U$.
    Assume $|\grad f_i|_{\R^N} \leq m$ and $|\Hess f_i|_2 \leq K$ for $i \in \{1,\dots,k\}$ on $U$.
    Denote by $\lambda$ the length of the smallest bottleneck of $M$.
    Then:
    \[
        \lambda \geq \frac{1}{2 \Cr{const:U-to-R^N-norm-equivalence} k K (m \Cr{const:U-to-R^N-norm-equivalence} k+1)},
    \]
    where $\Cr{const:U-to-R^N-norm-equivalence}$ is the constant defined in \cref{proposition:U-norm-equivalence}.
\end{proposition}

\begin{proof}
    Let $p,q \in M$ be two points connected by a smallest bottleneck, i.e. $|p-q|_{\R^N}=\lambda$ and the straight line joining $p$ and $q$ is normal to $M$ at those points.
    Denote by $\gamma : [0,\lambda] \rightarrow \R^N$ the unit speed straight line from $p$ to $q$.

    Write $u_i := \grad f_i$ as a shorthand.
    By definition of bottleneck, we have that $\gamma'(0) \perp T_p M$ and $\gamma'(\lambda) \perp T_q M$, i.e. there exists $\alpha_1,\dots,\alpha_k, \beta_1,\dots, \beta_k \in \R$ such that
    \[
        \gamma'(0)=\sum_{i=1}^k \alpha_i u_i(p)
        \text{ and }
        \gamma'(\lambda)=\sum_{i=1}^k \beta_i u_i(q).
    \]
    Then for $i \in \{1,\dots,k\}$:
    \begin{align}
        \label{equation:alpha-beta-estimate}
        \begin{split}
        |\alpha_i|
        &\leq
        \sum_{j=1}^k |\alpha_j|
        =|\gamma'(0)|_U
        \leq
        \Cr{const:U-to-R^N-norm-equivalence}
        |\gamma'(0)|_{\R^N}
        =
        \Cr{const:U-to-R^N-norm-equivalence},
        \\
        |\beta_i|
        &\leq \sum_{j=1}^k |\beta_j|
        =|\gamma'(0)|_U \leq
         \Cr{const:U-to-R^N-norm-equivalence}
        |\gamma'(0)|_{\R^N}
        = \Cr{const:U-to-R^N-norm-equivalence},
        \end{split}
    \end{align}
    and
    \begin{align}
        \label{equation:Ui-estimate}
        \begin{split}
        |u_i(q)-u_i(p)|_1
        &=
        \sum_{j=1}^k
        \left|
        \frac{\partial f_i}{\partial x_j}(q)
        -
        \frac{\partial f_i}{\partial x_j}(p)
        \right|
        \\
        &=
        \lambda \cdot
        \sum_{j=1}^k
        \left|
        \gamma'
        \left(
        \frac{\partial f_i}{\partial x_j}( \gamma(\theta_{ji}))
        \right)
        \right|
        \quad \quad \text{ for some $\theta_{ji} \in (0,\lambda)$}
        \\
        &=
        \lambda \cdot 
        \sum_{j=1}^k
        \left|
        \gamma'(\theta_{ji}) \cdot \Hess f_i \cdot e_j
        \right|
        \\
        &\leq
        \lambda \cdot
        \sum_{j=1}^k
        \underbrace{|\gamma'|_{\R^N}}_{=1}
        \cdot
        |\Hess f_i |_2
        \cdot
        \underbrace{|e_j|_{\R^N}}_{=1},
        \end{split}
    \end{align}
    where we used the mean value theorem in the second step,
    and we used the Cauchy-Schwarz inequality and definition of $|\cdot|_2$-norm for matrices in the last step.
    Thus,
    \begin{align*}
        \sum_{i=1}^k |\beta_i - \alpha_i|
        &=
        \left|
        \sum_{i=1}^k (\beta_i-\alpha_i) u_i(p)
        \right|_U
        \\
        &\leq
        \Cr{const:U-to-R^N-norm-equivalence}
        \left|
        \sum_{i=1}^k (\beta_i-\alpha_i) u_i(p)
        \right|_{\R^N}
        \\
        &\leq
        \Cr{const:U-to-R^N-norm-equivalence}
        \left(
        \underbrace{
        \left|
        \sum_{i=1}^k 
        \alpha_i u_i(p)-\beta_i u_i(q)
        \right|_{\R^N}
        }_{=|\gamma'(0)-\gamma'(\lambda)|_{\R^N}=0}
        +
        \left|
        \sum_{i=1}^k
        \beta_i u_i(q)-\beta_i u_i(p)
        \right|_{\R^N}
        \right)
        \\
        &\leq
        \Cr{const:U-to-R^N-norm-equivalence}
        |\beta|_{\infty}
        \cdot
        \left(
        \sum_{i=1}^k
        |u_i(q)-u_i(p)|_{1}
        \right)
        \\
        &\leq
        \Cr{const:U-to-R^N-norm-equivalence}^2 k^2 \lambda K,
    \end{align*}
    where the first step is the definition of $|\cdot|_U$;
    in the second step we used \cref{proposition:U-norm-equivalence}; in the third step we added and subtracted $\sum \beta_iu_i(q)$ and applied the triangle inequality;
    in the fourth step we used the norm equivalence $|\xi|_{\R^N} \leq |\xi|_1$ for all $\xi \in \R^N$;
    and in the last step we used \cref{equation:alpha-beta-estimate} to estimate $|\beta|_\infty$ and used \cref{equation:Ui-estimate} to estimate the sum.

    Define $\psi_i(t):=\alpha_i + \frac{t}{\lambda} (\beta_i-\alpha_i)$ and $F(t):=\sum_i \psi_i(t) f_i(\gamma(t))$. Observe that $F(0)=0=F(\lambda)$, so by Rolle's theorem, there exists $\theta \in (0,\lambda)$ such that
    \begin{align*}
        \label{equation:rolle-theorem-psi-i}
        0
        =
        F'(\theta).
    \end{align*}
    Since $F'(0)= |\gamma'(0)|^2_{\R^N} =1$, 
    \begin{align*}
        1
        &=
        F'(0) - F’(\theta) 
        \\
        &=
        \theta
        \frac{\d^2}{\d t^2}
        \bigg\rvert_{t=\xi}
        \sum_i \psi_i(t) f_i(\gamma(t))
        \quad
        \text{ for some }
        \xi \in (0,\theta)
        \\
        &\leq
        \theta \sum_{i=1}^k
        \left(        \underbrace{\psi_i''(\xi)}_{=0}
        f_i(\gamma(\xi))
        +
        2 \underbrace{|\psi'(\xi)|}
        _{\leq |\beta_i-\alpha_i|/\lambda} 
        \underbrace{|\langle u_i(\gamma(\xi)), \gamma'(\xi) \rangle|}
        _{\leq m}
        +    \underbrace{|\psi_i(\xi)|}_{\leq |\alpha_i|+|\beta_i|}
        \underbrace{|\langle \nabla_{\gamma'(\xi)} \nabla f_i, \gamma'(\xi) \rangle|}
        _{\leq K}
        \right)
        \\
        &\leq
        \theta
        \left(
        2\Cr{const:U-to-R^N-norm-equivalence}^2 k^2 Km
        +
        2\Cr{const:U-to-R^N-norm-equivalence} kK
        \right)
        \leq
        2 \lambda \Cr{const:U-to-R^N-norm-equivalence} k K (m \Cr{const:U-to-R^N-norm-equivalence} k+1).
    \end{align*}

    Thus, $\lambda \geq 1/(2 \Cr{const:U-to-R^N-norm-equivalence} k K (m \Cr{const:U-to-R^N-norm-equivalence} k+1))$.
\end{proof}

\subsection{Using $\det g$ for a lower bound of the reach}
\label{subsection:det-g-for-reach}

A lower bound for the reach is now an immediate corollary of \cref{lemma:det-g-for-second-FF-bound,proposition:det-g-for-smallest-bottleneck,theorem:reach-characterisation}, which we state here for convenience:

\begin{corollary}
    Let $M=Z(f_1,\dots,f_k) \subset U \subset \R^N$ for a convex set $U$ and let $u_i := \grad f_i$.
    We define the matrix $g\in C^\infty(\R^N, \R^{k \times k})$ whose entries are,
    \[
    (g_{ij}) := 
     \< u_i, u_j \>.
    \]
    Assume the following bounds,
    \begin{align*}
    \|{g}_1
    &\leq 
    \Cl{g-upper-bound},
    \\
    |\det g|
    &\geq
    \Cl{det-g-lower-bound},
    \\
    |\Hess f_i|_2 
    &\leq
    \Cl{Hess-upper-bound},
    \\
    | \grad f_i |_{\R^N}
    &\leq
    \Cl{grad-upper-bound}
    \end{align*}
    hold on $U$.
    Then
    \begin{equation}
    \tau
    \geq
    \min
    \left\{
    k^{-4-(k-1)/4}
    \frac{\Cr{det-g-lower-bound}^{1/2}}{\Cr{g-upper-bound}^{(k-1)/2} \Cr{Hess-upper-bound} },
    \quad
    \frac{1}{2} \cdot
    \frac{1}{(2 k \Cr{Hess-upper-bound} \Cr{const:U-to-R^N-norm-equivalence} (\Cr{grad-upper-bound} \Cr{const:U-to-R^N-norm-equivalence} k+1))}
    \right\},
    \end{equation}
    where $\Cr{const:U-to-R^N-norm-equivalence}$ was defined in \cref{proposition:U-norm-equivalence}.
\end{corollary}

\section{The quantitative tubular neighbourhood theorem}
\label{section:tubular-nbhd}

\begin{theorem}
    \label{theorem:tubular-nbhd-theorem}
    Let $M^n \subset \mathbb{R}^N$ be a closed embedded submanifold without boundary
    and let $\tau$ be its reach.
    For
    \begin{align*}
        V^\tau
        &:=
        \{(x,v) \in NM : |v|_{\R^N} < \tau \} \text{ and } U^\tau_M := \{p \in \R^n: d(p,M)< \tau \}
    \end{align*}
    the endpoint map,
\begin{align}
    \begin{split}
    E: V^\tau & \rightarrow U^\tau_M \\
      (x,v)      &\mapsto x+v,
    \end{split}
\end{align} is a diffeomorphism.
\end{theorem}

\begin{proof}
    \leavevmode
    \begin{enumerate}
        \item 
        $E$ is injective.

        Assume $(x,v) \neq (y,w) \in V^\tau$ and let $z := E(x,v) = E(y,w)$. If $x=y$, then clearly, $v=w$, so we can assume without loss of generality that $x\neq y$ and $|v| \leq |w|$.
        Denote by $B$ the open ball with radius $|w|$ around $z$, so that $B$ is tangential to $M$ at $y$.

        We claim that $\partial B \cap M = \{y\}$.
        Otherwise, there exists $u \in \partial B \cap M$ with $u \neq y$ and
        \[
        \tau
        \leq
        \frac{|u-y|^2}{2d(u-y, T_y M)}
        =
        |w|
        <
        \tau.
        \]
        In the first inequality, we used \cite[Theorem 4.18]{Federer1959}, and for the equality, we used that the fraction can be interpreted as the radius of a ball tangent to $M$ at $y$ passing through $u$, see \cite[Figure 1]{Aamari2019}.
        This is a contradiction, so the claim holds.

        Let $C$ be the connected component of $M$ containing $x$. There are two possibilities: $C \cap \partial B = \emptyset$ or $C \cap \partial B = \{y\}$.
        
 In the first case, let $\gamma: \R \rightarrow C$ be a closed unit speed geodesic in $C$.
        By \cref{proposition:curve-with-distance-maximum-has-large-curvature} we have that there exists $t_0 \in \R$ such that
        \[
        |\gamma''(t_0)|_{\R^N} \geq \frac{1}{|w|} > \frac{1}{\tau},
        \]
        which is a contradiction to \cref{theorem:reach-characterisation}.

In the second case, let $\gamma$ be a unit speed geodesic joining $x$ and $y$ (since $M$ is complete and $C$ is connected, it exists by Hopf-Rinow's theorem). By a translation of $M$ in $\R^N$ we can assume that the ball $B$ is centred at $z=0$. Consider the function $u(t) = (|\gamma(t)|^2) '$, since it is a smooth function vanishing at $x = \gamma(t_1)$ and $y=\gamma(t_2)$, there is a point $t_0 \in [t_1,t_2]$ such that $u'(t_0)=0$, that is,
\begin{equation*}
    0 = u(t_0)' = 2 \langle \gamma', \gamma' \rangle + 2 \langle \gamma''(t_0), \gamma(t_0) \rangle \geq 2 - 2 |\gamma''(t_0)|_{\R^N} |\gamma(t_0)|_{\R^N}.
\end{equation*}
Therefore,
\begin{equation*}
    |\gamma''(t_0)|_{\R^N} \geq \frac{1}{|\gamma(t_0)|_{\R^N}} > \frac{1}{\tau}
    \end{equation*}
        which is a contradiction to \cref{theorem:reach-characterisation}.

        \item 
        $E$ is surjective.

        Given $y \in U^\tau_M$, let $x \in M$ be the unique closest point to $y$ in $M$.
        We claim that $(x, y-x) \in N_yM$.
        To show this, we must check that $\<y - x, v \>=0$ for all $v \in T_x M$.
        Fix $v \in T_x M$.
        Because $M$ has no boundary, there exists $\gamma:(-\epsilon,\epsilon) \rightarrow M$ such that $\gamma(0)=x$ and $\gamma'(0)=v$.
        By assumption, $x$ is the unique closest point to $y$, thus
        \begin{align*}
            0
            &=
            \frac{\d}{\d t}|_{t=0} d(\gamma(t),y)^2
            =
            2 \< v, x-y \>,
        \end{align*}
        which proves the claim.
        Thus $(x, y-x) \in N_yM$, and therefore $E(x, y-x)=y$, which shows surjectivity.

        \item 
        $E$ is a diffeomorphism.

Since $E$ is a bijective, smooth map, it is enough to show that it is a local diffeomorphism (\cite[Proposition 4.6 (f)]{lee2012smooth}. 
        In other words, it suffices to show that $\d_{(x,v)} E$ is invertible for all $(x,v) \in V^\tau$. As with the standard tubular neighbourhood theorem, one can check that $\d_{(x,0)} E$ maps $T_{(x,0)} (M\times \{0\})$ bijectively to $T_x M \subset T_x \R^n$ and  $T_{(x,0)} (\{x\} \times N_xM)$ to $T_xNM \subset T_x \R^{N-n}$.
        For $v\neq 0$, by \cite[p.34]{milnor1963morse} we have that $\d_{(x,v)} E$ is not invertible precisely when $e=E(x,v)$ is a focal point of $(M,x)$, that is, $e=x+K_i^{-1} \frac{e-x}{|e-x|_{\R^N}}$ for some $i \in \{1,\dots,n\}$, where $K_i^{-1}$ are the  principal radii of curvature of $M$ at $x$ in the direction $v$. This equation holds only when $|e-x|_{\R^N}
        =
        K_i^{-1}$, which by Theorem \ref{theorem:reach-characterisation} is larger or equal than the reach.
        \qedhere

    \end{enumerate}
\end{proof}

\bibliographystyle{apalike}
\bibliography{library}

\noindent{\small\sc Imperial College London, Department of Mathematics, 180 Queen's Gate, South Kensington, London SW7 2RH, the United Kingdom 

\noindent E-mail: {\tt \href{mailto:d.platt@imperial.ac.uk}{d.platt@imperial.ac.uk}}

\vskip 8pt

\noindent{\small\sc 4i Intelligent Insights, Tecnoincubadora Marie Curie, PCT Cartuja, 41092 Sevilla, Spain

\noindent E-mail: {\tt \href{mailto:r.sanchez@4i.ai}{r.sanchez@4i.ai}}

\end{document}